\documentclass[12pt,reqno]{amsart}
\usepackage{amsmath, amssymb, amsfonts, amsthm, comment}
\usepackage{mathtools,xcolor}
\usepackage{hyperref}
\usepackage[all,cmtip]{xy}
\usepackage{tkz-euclide}
\usepackage{tikz-cd}
\usepackage[utf8]{inputenc}
\allowdisplaybreaks
\newtheorem{thm}{Theorem}[section]
\newtheorem*{thm*}{Theorem}
\newtheorem*{coro*}{Corollary}
\newtheorem{lem}[thm]{Lemma}
\newtheorem{prop}[thm]{Proposition}
\theoremstyle{definition}
\newtheorem{dfn}[thm]{Definition}
\newtheorem{exple}[thm]{Example}

\theoremstyle{plain}
\newtheorem{cor}[thm]{Corollary}
\numberwithin{equation}{section}
\usepackage[a4paper, top = 1.2in, bottom = 1.2in, left = 1.2in, right = 1.2in]{geometry}
\numberwithin{equation}{section}

\newcommand{\N}{\mathbb{N}}
\newcommand{\Q}{\mathbb{Q}}

\newcommand{\Z}{\mathbb{Z}}

\newcommand{\F}{\mathbb{F}}

\newcommand{\mcS}{\mathcal{S}}

\newcommand{\mfa}{\mathfrak{a}}

\newcommand{\mfp}{\mathfrak{p}}
\newcommand{\mfl}{\mathfrak{l}}

\newcommand{\SL}{\mathrm{SL}}
\newcommand{\GL}{\mathrm{GL}}

\newcommand{\End}{\mathrm{End}}
\newcommand{\Aut}{\mathrm{Aut}}
\newcommand{\Frob}{\mathrm{Frob}}
\newcommand{\tr}{\mathrm{tr}}
\newcommand{\un}{\mathrm{un}}

\newcommand{\Spec}{\mrm{Spec}}
\newcommand{\sep}{\mrm{sep}}
\newcommand{\trace}{\mrm{trace}}
\newcommand{\Nr}{\mrm{Nr}}

\newcommand{\Gal}{\mrm{Gal}}
\newcommand{\Ima}{\mrm{Im }}
\newcommand{\spn}{\mrm{span }}

\newcommand{\lra}{\longrightarrow}
\newcommand{\ra}{\rightarrow}

\newcommand{\mrm}[1]{\mathrm{#1}}

\def\1{1\!\!1}

\newcommand{\psmat}[4]{\bigl( \begin{smallmatrix} #1 & #2 \\ #3 & #4 \end{smallmatrix} \bigr)}

\title[On the surjectivity of Galois representations]{On the surjectivity of Galois representations attached to Drinfeld $A$-modules of rank $2$}
\author[N. Kumar]{Narasimha Kumar}
\email{narasimha@math.iith.ac.in}
\address{
	Department of Mathematics \\
	Indian Institute of Technology Hyderabad\\
	Kandi, Sangareddy - 502284\\
	INDIA.
}
\author[D. Shit]{Dwipanjana Shit}
\email{ma22resch01001@iith.ac.in}
\address{
	Department of Mathematics \\
	Indian Institute of Technology Hyderabad\\
	Kandi, Sangareddy - 502284\\
	INDIA.
}

\begin{document}
	\begin{abstract}
        Let $\F_q$ be a finite field with $q$ elements, where $q$ is a prime power and let $A:= \F_{q}[T]$ and $F:=\F_{q}(T)$. By~\cite{PR09}, the adelic image of the Galois representation attached to a generic Drinfeld $A$-module $\varphi$ of rank $2$ over $F$ satisfying $\End_{\bar{F}}(\varphi)=\varphi(A)$ is open, and determining when it is surjective remains a subtle problem. To resolve this question, in this article, we study the $\mfp$-adic surjectivity of the Galois representations attached to $\varphi$, where $\mfp \in \Omega_A:= \Spec(A) \setminus \{ (0) \}$. There are two directions to investigate this problem: one by fixing the prime $\mfp$, and the other by fixing $\varphi$.
        
        In the horizontal direction, for a fixed $\mfp \in \Omega_A$, we give explicit and easily verifiable conditions on Drinfeld $ A$-modules $\varphi$ of rank $2$ over $F$ which ensure the surjectivity of the $\mfp$-adic Galois representation $\rho_{\varphi,\mfp}$. This work extends~\cite{Ray24} to arbitrary $\mfp$, and, for $ \mfp=(T)$, recovers the results of~\textit{loc. cit.} under weaker hypotheses.

        In the vertical direction, we show that for a fixed Drinfeld $A$-module $\varphi$ of rank $2$ over $F$, whose coefficients satisfy certain congruence and valuation conditions, the $\mfp$-adic Galois representation $\rho_{\varphi,\mfp}$ is surjective for all $\mfp \in \Omega_A$. This recovers the example of \cite{Zyw11} and yields new examples beyond those considered in \cite{Zyw25}. As a consequence, we obtain the surjectivity of the associated adelic Galois representation.
	\end{abstract}
	\maketitle

    \section{Introduction, Literature, and Statement of Results}
    Galois representations are one of the fundamental objects of study in arithmetic geometry. They typically arise from data associated with elliptic curves over number fields, elliptic modular forms, Drinfeld modular forms, etc. The trace and determinant of these representations at Frobenius elements capture essential information about the arithmetic properties of the respective object. In this article, we are interested in studying the surjectivity of $\mfp$-adic Galois representations attached to Drinfeld $A$-modules of rank $2$ over $F$, where $\mfp \in \Omega_{A}:= \Spec(A) \setminus \{ (0) \}$.
	
	In the first part of the article, for a fixed $\mfp \in \Omega_A$, we give sufficient conditions on the coefficients of Drinfeld $A$-modules of rank $2$ over $F$ ensuring 
    the surjectivity of the $\mfp$-adic Galois representation $\rho_{\varphi,\mfp}$ attached to $\varphi$ (Theorem~\ref{mfp_adic_surjectivity_main_theorem_1}). 
    
    In the second part of the article, for a fixed Drinfeld $A$-module $\varphi$ of rank $2$ over $F$ that satisfies some congruence conditions and valuations on the coefficients, we show the surjectivity of the $\mfp$-adic Galois representation attached to $\varphi$ for all $\mfp \in \Omega_A$ (Theorem~\ref{for_all_mfp_adic}). As a consequence, we show the surjectivity of the adelic Galois representation attached to $\varphi$ (Theorem~\ref{adelic_surjectivity}).
	
\subsection{Literature:}
	In this section, we first discuss the relevant literature for the elliptic curves over $\Q$. Let $E$ be an elliptic curve over $\Q$. For every $n\in \N$, there exists a Galois representation $\bar{\rho}_{E,n}:\Gal(\bar{\Q}/\Q)\ra \Aut(E[n])\cong \GL_{2}(\Z/n\Z)$, where $E[n]$ is the $n$-torsion of $E(\bar{\Q})$[cf. \cite{Sil09} III.7]. For a prime $p \in \mathbb{P}$, the $p$-adic Tate module $T_p(E):=\varprojlim_{n}\ E[p^n]$ admits a natural
	$\Gal(\bar{\Q}/\Q)$-action which gives a representation
    \begin{equation*}
        \rho_{E,p}: \Gal(\bar{\Q}/\Q)\ra \Aut(T_p(E))\cong \GL_{2}(\Z_{p}).
    \end{equation*}
	As a consequence, we obtain the adelic Galois representation
    \begin{equation*}
        \rho_{E}:\Gal(\bar{\Q}/\Q)\ra \prod_{p}^{}\GL_{2}(\Z_{p})\cong\GL_{2}(\widehat{\Z}).
    \end{equation*}
	The famous ``Open Image Theorem'' of Serre states that
	\begin{thm}[\cite{Ser72}]
		If $E$ is an elliptic curve over $\Q$ without complex multiplication, then the associated adelic Galois representation $\rho_{E}$ has an open image. In particular, $[\GL_{2}(\widehat{\Z}): \Ima(\rho_{E})] < \infty.$
	\end{thm}

     In the context of Drinfeld $A$-modules over $F$, Pink and R\"utsche~\cite{PR09} studied the adelic Galois representations attached to Drinfeld $A$-modules over $F$, the function field analogues of elliptic curves, and proved the following result.
	\begin{thm}[\cite{PR09}]
		Let $\varphi$ be a Drinfeld $A$-module of rank $r$ over $F$ such that $\varphi$ is generic.
		If $\End_{\bar{F}}(\varphi)=\varphi(A)$, then the image of the associated adelic Galois representation $\rho_{\varphi}(G_{F})$ is open in $\GL_{r}(\widehat{A})$. Equivalently, $[\GL_{r}(\widehat{A}):\rho_{\varphi}(G_{F})]<\infty.$
	\end{thm}

\subsection{Basic Question:}
    A natural question in both settings is to determine when the image has index $1$; equivalently, when the associated adelic Galois representation is surjective. For elliptic curves over $\Q$, Serre proved in~\cite{Ser72} that the image always has index at least $2$ in $\GL_{2}(\widehat{\Z})$. However, Greicius showed in~\cite{Gre10} that there exists a number field $K$ and an elliptic curve $E/K$ for which the adelic Galois representation is surjective. In contrast, for Drinfeld $A$-modules over $F$, there are known examples with surjective adelic representations.
   
    In rank $1$, the adelic surjectivity of the Galois representation attached to the Carlitz $A$-module follows from the work of Hayes~\cite{Hay74}. However, for an arbitrary Drinfeld $A$-module of rank $1$ over $F$, in~\cite{Gek16}, Gekeler showed that the index of the image of the associated adelic Galois representation in $\widehat{A}^\times$ is a divisor of $q-1$. Note that every Drinfeld $A$-module of rank $1$ over $F$ is a twist of the Carlitz $A$-module.
        
    For rank $2$, in~\cite{Zyw11}, Zywina constructed a Drinfeld $A$-module of rank $2$ over $F$ defined by $\varphi_{T}=T+\tau-T^{q-1}\tau^2$ having surjective adelic Galois representations if $q \geq 5$ is an odd prime power. In a subsequent work~\cite{Zyw25}, he introduced four subsets of $A^2$ and showed that Drinfeld $A$-modules corresponding to elements in their intersection have surjective $\mfp$-adic Galois representations for every $\mfp\in \Omega_{A}$. 

\subsection{Our contribution:}
    In~\cite{Ray24}, under suitable conditions on the coefficients, the author proved the surjectivity of the $(T)$-adic Galois representation associated to a Drinfeld $A$-module of rank $2$ over $F$. More precisely:
    \begin{thm*}[\cite{Ray24}, Theorem 3.1]
    \label{Ray24a_Main_Result_1}
		Let $q\geq 5$ be odd and $\eta\in \F_{q}^{\times}$ be a non-square, and let $a_{1}, a_{2}$ be distinct non-zero elements in $\F_{q}^{\times}$. Let $\varphi$ be a Drinfeld $A$-module of rank $2$ over $F$ defined by $\varphi_{T}=T+g_{1}\tau+g_{2}\tau^2 \in A\{\tau\}$ such that
			\begin{itemize}
				\item $\nu_{(T)}(g_1)=0$, $\nu_{(T-a_{1})}(g_{1})\geq 1$, and $\nu_{(T-a_2)}(g_{1})=0$,
				\item $\nu_{(T)}(g_2)=0$, $g_{2}\equiv -a_{1}\eta^{-1}\pmod{(T-a_{1})}$ and $\nu_{(T-a_{2})}(g_{2})=1$.
			\end{itemize} 
		Then, the representation $\rho_{\varphi,(T)}$ is surjective. %Moreover, the set of such Drinfeld $ A$-modules of rank $2$ has lower density at least $ q^{-7}$. 
    \end{thm*}
    Now, in the horizontal direction, we state the main result of this article for an arbitrary $\mfp\in\Omega_{A}$.
    \begin{thm*}[Theorem~\ref{mfp_adic_surjectivity_main_theorem_1} in the text]
        Let $q=p^e\ge 5$ be an odd prime power and $\mfp\in \Omega_A$. Let $\varphi$ be a Drinfeld $A$-module of rank $2$ over $F$ defined by
        \begin{equation*}
            \varphi_T= T + g_1\tau - g_2^{\,q-1}\tau^2 \in A\{\tau\}.
        \end{equation*}
        Assume that $g_1$ and $g_2$ satisfy the following conditions:
        \begin{enumerate}
           \item There exists $c\in \F_q$ with $(T-c)\in \Omega_A\setminus\{\mfp\}$ such that $\nu_{(T-c)}(g_2)=0$, and the polynomial
           \begin{equation*}
               x^2 - g_1(c)x + (T-c)
           \end{equation*}
           is irreducible over $\F_{\mfp}$,
           \item There exists $\mfl\in\Omega_{A}\setminus\{\mfp,(T-c)\}$ such that $ p\nmid \nu_{\mfl}(g_2)$ and $\nu_{\mfl}(g_1)=0$.
       \end{enumerate}
       Then, the associated $\mfp$-adic Galois representation $\rho_{\varphi,\mfp}$ is surjective.
    \end{thm*}
    As an immediate consequence, we obtain the following corollary, which recovers~\cite[Theorem 3.1]{Ray24} in the case $\mfp=(T)$, but under comparatively simpler and easily verifiable conditions than those in~\cite[Theorem 3.1]{Ray24}.
    
    \begin{coro*} 
        Let $q=p^e\geq 5$ be an odd prime power and let $\mfp\in \Omega_{A}$ be such that there exists $c\in \F_q$ with $c-T$ being non-square modulo $\mfp$. Let $\varphi$ be a Drinfeld $A$-module of rank $2$ over $F$ defined by
        \begin{equation*}
            \varphi_{T}=T+g_{1}\tau-g_{2}^{q-1}\tau^2 \in A\{\tau\}.
        \end{equation*}
        Assume that $g_1$ and $g_2$ satisfy the following conditions:
        \begin{enumerate}
            \item $\nu_{(T-c)}(g_2)=0$ and $\nu_{(T-c)}(g_1)\geq 1$,
            \item There exists $\mfl\in\Omega_{A}\setminus\{\mfp,(T-c)\}$ such that $ p\nmid \nu_{\mfl}(g_2)$ and $\nu_{\mfl}(g_1)=0$.
        \end{enumerate}
        Then the associated $\mfp$-adic Galois representation $\rho_{\varphi,\mfp}$ is surjective.
    \end{coro*}

    Now, in the vertical direction, we state the main result of this article. Let $\Lambda_{A}$ be the set of pairs $(\mfl,r_1)\in\Omega_{A}\times\F_{q}$ such that the polynomial $x^2-r_1x+(T-c)$ is irreducible over $\F_{\mfl}$ for some $c\in \F_{q}$.
    \begin{thm*}[Theorem~\ref{for_all_mfp_adic} in the text]
        Let $q=p^e\geq 5$ be an odd prime power and $(\mfl,r_{1})\in \Lambda_{A}$. Let $\varphi$ be a Drinfeld $A$-module of rank $2$ over $F$ defined by
        \begin{equation*}
            \varphi_{T}=T+g_{1}\tau-\mfl^{q-1}\tau^2\in A\{\tau\},
        \end{equation*}
        with $g_1\equiv r_1\pmod{T^q-T}$ and $\nu_{\mfl}(g_1)=0$. Then the $\mfp$-adic Galois representation 
        \begin{equation*}
            \rho_{\varphi,\mfp}:G_{F}\longrightarrow \GL_2(A_{\mfp})
        \end{equation*}
		is surjective for all $\mfp\in \Omega_{A}$.
    \end{thm*} 
    Note that the Drinfeld $A$-modules of rank $2$ over $F$ constructed in Theorem~\ref{for_all_mfp_adic} not only generalise the example of Zywina~\cite{Zyw11} (corresponding to $(T,1)\in \widehat{\Lambda}_{A}$), but also yield new examples that lie outside the scope of~\cite{Zyw25}.

    As a consequence, we obtain the surjectivity of the adelic Galois representation. To state our result, we define $\widehat{\Lambda}_{A}:=\left\{(\mfl,r_{1})\in\Lambda_{A}: r_1\in \F_{q}^{\times} \right\}.$
    \begin{thm*}[Theorem~\ref{adelic_surjectivity} in the text]
        Let $q=p^e\geq 5$ be an odd prime power and $(\mfl,r_{1})\in \widehat{\Lambda}_{A}$. Let $\varphi$ be a Drinfeld $A$-module of rank $2$ over $F$ defined by
        \begin{equation*}
            \varphi_{T}=T+g_{1}\tau-\mfl^{q-1}\tau^2\in A\{\tau\},
        \end{equation*}
        with $g_1\equiv r_1\pmod{T^q-T}$ and $\nu_{\mfl}(g_1)=0$. Then the adelic Galois representation 
       \begin{equation*}
           \rho_{\varphi}:G_F\longrightarrow \GL_{2}(\widehat{A})
       \end{equation*}
		is surjective.
    \end{thm*}
 
    The surjectivity of Galois representations attached to Drinfeld modules provides explicit realisations of finite and infinite profinite general linear groups as Galois groups over global function fields, thereby contributing to the inverse Galois problem (cf. Corollaries ~\ref{cor_padic_inverse_galois},~\ref{cor:adelic_inverse_galois}). The adelic surjectivity implies the absence of non-zero $F$-rational torsion points attached to Drinfeld modules (cf. Corollary~\ref{nonzero_F_rational}). These connections motivate the search for explicit sufficient conditions that ensure surjectivity, as pursued in the present article.

	\section{Preliminaries}
    Throughout this article, we use the following notation. Let $q=p^e$ be an odd prime power. Set $A:=\F_{q}[T]$ with field of fractions $F:=\F_{q}(T)$, $G_{F}:=\Gal(F^{\sep}/F)$, where $F^{\sep}$ is the separable closure of $F$ in $\bar{F}$. Let $\widehat{A} :=\varprojlim_{\mfa \lhd A}A/\mfa$, where $\mfa$ runs over all non-zero ideals of $A$, denote the profinite completion of $A$. For a commutative ring $R$ with unity, $R^{\times}$ denotes the set of all units in $R$.
    
    To avoid notational complexity, we shall use $\mfa\subseteq A$ to denote both the generator and the ideal it generates. Let $\mfp\in \Omega_{A}$, $A_{\mfp}$ denote the completion of $A$ with respect to $\mfp$ with field of fractions $F_{\mfp}$, which is complete with respect to the normalized discrete valuation $\nu_p$, with the residue field $\F_{\mfp} := A_\mfp/\mfp A_\mfp(\cong A/\mfp)$, and absolute value $|\alpha|_{\mfp}:=|\F_{\mfp}|^{-\nu_{\mfp}(\alpha)}$. Also, $G_{F_{\mfp}}$ denotes the Galois group $\Gal(F_{\mfp}^{\sep}/F_{\mfp})$.

	\subsection{Drinfeld Modules} Let $K$ be an $A$-field, i.e., there exists an $\F_q$-algebra homomorphism $\gamma:A\ra K$. Let $K\{\tau\}$ be the ring of twisted polynomials, i.e., the elements of $K\{\tau\}$ are polynomials in the indeterminate $\tau$ with coefficients in $K$ that satisfy the usual addition of polynomials and the multiplication rule $(a\tau^i)(b\tau^j):=ab^{q^i}\tau^{i+j}$. For a twisted polynomial $f(\tau)=a_{h}\tau^h+a_{h+1}\tau^{h+1}+\cdots+a_{d} \tau^d \in K\{\tau\}$ with $0\leq h\leq d$ and $a_{h} a_{d}\neq 0$, define $\mathrm{ht}_{\tau}(f):=h$ and $\deg_{\tau}(f):=d.$
	
	\begin{dfn}[Drinfeld module]
		\label{Drinfeld-Definition}
		A Drinfeld $A$-module of rank $r\geq 1$ over $K$ is an $\F_{q}$-algebra homomorphism
		\[\begin{aligned}
			\varphi: A& \lra K\{\tau\} \\
			a & \longmapsto \varphi_a=\gamma(a)+g_1(a)\tau+\cdots+g_{r \deg_T(a)}(a)\tau^{r \deg_T(a)},
		\end{aligned}\]
		with $g_{r\deg_T(a)}(a)\neq 0$ if $a \neq 0$. Furthermore, $\varphi$ is referred to as generic if $\ker(\gamma)=0$.
 
	\end{dfn}
	Since $\varphi$ is an $\F_q$-algebra homomorphism, $\varphi_{T}$ uniquely determines $\varphi$. In fact, for $a\in A$, the coefficients $g_i(a)$ in $\varphi_{a}$ can be explicitly described in terms of the coefficients of $a$ and $\varphi_T$ (cf. Proposition~\ref{expression_for_varphi_a} in the Appendix). 
    
    Conversely, to define a Drinfeld $A$-module $\varphi$ of rank $r$ over $K$, it is enough to choose $g_1,\ldots, g_r\in K$ with $g_r \neq 0$ and define
    \begin{equation*}
        \varphi_{T}=\gamma(T)+g_{1}\tau+\cdots+g_{r}\tau^r,
    \end{equation*}
	and extend uniquely to an $\F_q$-algebra homomorphism $\varphi$ from $A$ to $K\{\tau\}$. Clearly, $\varphi$ is determined by $(g_1,\ldots, g_r)$.
    
	\begin{exple}[Carlitz module]
        A classical example of a Drinfeld $A$-module over $F$ of rank $1$, called the Carlitz module and denoted by $C$, is defined by $C_T = T + \tau$.
	\end{exple}
	
	\begin{dfn}[Morphism]
        Let $\varphi$ and $\psi$ be Drinfeld $A$-modules over $F$. A morphism $u:\varphi\ra \psi$ is a polynomial $u\in F\{\tau\}$ such that $u\varphi_T=\psi_T u$. We say $\varphi$ and $\psi$ are isomorphic over $F$ if there exists $c\in F^\times$ with $\psi_T=c\varphi_T c^{-1}$.
	\end{dfn}

\subsection{Galois representations attached to Drinfeld modules}
	Before we discuss the Galois representations attached to a generic Drinfeld $A$-module $\varphi$ of rank $r$ over $F$, we need to understand the torsion subgroups of Drinfeld modules.
    
\subsubsection{Torsion subgroups attached to Drinfeld modules}
    For each $a\in A$, the element $\varphi_a$ corresponds to an $\F_q$-linear polynomial $\varphi_{a}(x)=\gamma(a)x+\sum_{i=1}^{r\deg_{T}(a)}g_i(a)x^{q^i}$. Then $F^{\sep}$ acquires a (twisted) $A$-module structure via $b\cdot \alpha:=\varphi_{b}(\alpha)$ for all $b\in A$ and $\alpha\in F^{\sep}$, where $\varphi_b(\alpha)$ denotes the evaluation of the $\F_q$-linear polynomial $\varphi_b(x)$ at $\alpha$. 
    
    For any non-zero polynomial $a\in A$, the $a$-torsion of $\varphi$ is defined by
    \begin{equation*}
        \varphi[a]:=\{\alpha\in F^{\sep}|a\cdot \alpha=0 \}= \{\alpha\in F^{\sep}|\varphi_{a}(\alpha)=0 \}.
    \end{equation*}
	Note that $\varphi[a]$ is a finite dimensional $\F_{q}$-vector space. In fact, $\varphi[a]$ is an $A$-submodule of $F^{\sep}$ since for any $b\in A$ and $\alpha\in \varphi[a]$, we have
    \begin{equation*}
        \varphi_{a}(\varphi_{b}(\alpha))=\varphi_{ab}(\alpha)=\varphi_{ba}(\alpha)=\varphi_{b}(\varphi_{a}(\alpha))=0.
    \end{equation*}
    For any non-zero ideal $\mfa\subseteq A$ with monic generator $a$, we define $\varphi[\mfa]:= \varphi[a]$. By~\cite[Corollary 3.5.3]{Pap23}, the module $\varphi[\mfa]\cong (A/\mfa)^r$, since $\varphi$ is generic. 
	
\subsubsection{Galois representations}
    The absolute Galois group $G_F$ acts on $\varphi[\mfa]$ by permuting the roots of $\varphi_a(x)$. This action commutes with the $A$-module structure of $\varphi[\mfa]$, since for $\sigma\in G_F$, $b\in A$, and $\alpha\in \varphi[\mfa]$, we have $\varphi_b(\sigma\alpha)=\sigma(\varphi_b(\alpha))$. Thus, $G_F$ acts via $A$-module automorphisms, giving rise to the mod-$\mfa$ Galois representation
    \begin{equation}\label{mod_p_rp}
        \bar{\rho}_{\varphi,\mfa}:G_F \longrightarrow \Aut_A(\varphi[\mfa]) \cong \GL_r(A/\mfa).
    \end{equation}
    For $\mfp\in \Omega_{A}$, the natural maps 
    \begin{equation}
    \label{transition_maps}
       \varphi[\mfp^{n+1}]\xrightarrow{\times_{\mfp}} \varphi[\mfp^n], \quad \alpha\mapsto \mfp\cdot \alpha=\varphi_{\mfp}(\alpha)
	\end{equation}
	are surjective. Taking inverse limits, we obtain the $\mfp$-adic Tate module of $\varphi$
    \begin{equation*}
        T_{\mfp}[\varphi]:=\varprojlim_{n}\varphi[\mfp^n]\cong \varprojlim_{n}(A/\mfp^n)^r\cong A_{\mfp}^r.
    \end{equation*}
	The $G_F$-action on each $\varphi[\mfp^n]$ is compatible with the maps~\eqref{transition_maps}, hence induces the $\mfp$-adic Galois representation
	\begin{equation} \label{p_adic_rp}
		\rho_{\varphi,\mfp}:G_{F}\ra \Aut_{A_{\mfp}}(T_{\mfp}(\varphi))\cong \GL_{r}(A_{\mfp}).
	\end{equation}
	Combining these representations, we obtain the adelic Galois representation
	\begin{equation} \label{adelic}
		\rho_{\varphi}:G_F\longrightarrow \GL_{r}(\widehat{A})\cong \prod_{\mfp\in \Omega_{A}}\GL_{r}(A_{\mfp}).
	\end{equation}

	\subsection{Stable/Good reduction of a Drinfeld module \texorpdfstring{$\varphi$}:}
	For any Drinfeld $A$-module $\varphi$ over $F$ and $\lambda\in \Omega_{A}$, the localized Drinfeld $A$-module $\varphi_{\lambda}$ over $F_\lambda$ is defined by the composite
    \begin{equation*}
        \varphi_{\lambda}:A\xrightarrow{\varphi} F\{\tau\}\ra F_{\lambda}\{\tau\},
    \end{equation*}
	where the second map arises from the inclusion $F\hookrightarrow F_{\lambda}$.
	
	\begin{dfn}[Stable/Good reduction]
		The Drinfeld $A$-module $\varphi$ of rank $r$ over $F$ is said to have stable reduction at $\lambda$ if there exists a Drinfeld $A$-module $\psi$ of rank $r$ over $F_{\lambda}$ with $\psi_{T}\in A_{\lambda}\{\tau\}$, isomorphic to $\varphi_{\lambda}$ over $F_\lambda$, such that its reduction
        \begin{equation*}
            \bar{\psi}:A\lra \F_{\lambda}\{\tau\},
        \end{equation*}
		is a Drinfeld $A$-module over $\F_{\lambda}$. The rank of $\bar{\psi}$ is called the reduction rank of $\varphi$ at $\lambda$. If the reduction rank of $\varphi$ at $\lambda$ is equal to $r$, then we say $\varphi$ has good reduction at $\lambda$.
	\end{dfn}

%	\subsubsection{Height of a Drinfeld module:}
    Let $\varphi$ be a Drinfeld $A$-module of rank $r$ over $F$ with good reduction at $\lambda \in \Omega_A$. Suppose the reduced Drinfeld $A$-module over $\F_\lambda$ is $\bar{\psi}:A \longrightarrow \F_\lambda\{\tau\}.$ Since $\lambda$ lies in the kernel of the reduction map $\gamma:A\ra \F_{\lambda}$, the constant term of $\bar{\psi}_{\lambda}$ is $0$, and hence $\mathrm{ht}_{\tau}(\bar{\psi}_{\lambda})>0$. The height of $\bar{\psi}$ is defined by
    \begin{equation*}
        H_{\lambda}(\bar{\psi}):= \frac{\mathrm{ht}_{\tau}(\bar{\psi}_\lambda)}{\deg_{T}(\lambda)}.
    \end{equation*}
    We refer to this as the height of the Drinfeld $A$-module $\varphi$ modulo $\lambda$ and denote it by $H_\lambda$ if $\varphi$ is clear from the context. By~\cite[Lemma 3.2.11]{Pap23}, the integer $H_{\lambda} \in [1, r]$.

    \section{Horizontal $\mfp$-adic surjectivity of Drinfeld modules}

    In this section, for a fixed $\mfp \in \Omega_A$, we construct Drinfeld $A$-modules $\varphi$ of rank $2$ over $F$ with $\varphi_{T}\in A\{\tau\}$, whose coefficients satisfy certain $\mfp$-adic valuation conditions, for which the associated $\mfp$-adic Galois representation $\rho_{\varphi,\mfp}$ is surjective.

	\begin{thm}
    \label{mfp_adic_surjectivity_main_theorem_1}
        Let $q=p^e\geq 5$ be an odd prime power and let $\mfp\in \Omega_A$. Let $\varphi$ be a Drinfeld $A$-module of rank $2$ over $F$ defined by
        \begin{equation*}
            \varphi_T= T + g_1\tau - g_2^{\,q-1}\tau^2 \in A\{\tau\}.
        \end{equation*}
        Assume that $g_1$ and $g_2$ satisfy the following conditions:
        \begin{enumerate}
            \item There exists $c\in \F_q$ with $(T-c)\in \Omega_A\setminus\{\mfp\}$ such that $\nu_{(T-c)}(g_2)=0$, and the polynomial
            \begin{equation*}
                x^2 - g_1(c)x + (T-c)
            \end{equation*}
            is irreducible over $\F_{\mfp}$,
            \item There exists $\mfl\in\Omega_{A}\setminus\{\mfp,(T-c)\}$ such that $ p\nmid \nu_{\mfl}(g_2)$ and $\nu_{\mfl}(g_1)=0$.
        \end{enumerate}
        Then, the associated $\mfp$-adic Galois representation $\rho_{\varphi,\mfp}$ is surjective.
    \end{thm}
 
    The proof of Theorem~\ref{mfp_adic_surjectivity_main_theorem_1} is based on the following proposition by Pink and R\"utsche (cf. \cite[Proposition 4.1]{PR09}).
    \begin{prop}
    \label{Pink_R_cond_T_adic_sur}
        Let $\mfp\in \Omega_A$. Let $M$ be a closed subgroup of $\GL_{r}(A_{\mfp})$ such that
		$\det(M)=A_{\mfp}^{\times}$. Assume that
		$|\F_{\mfp}|\geq 4$. Suppose $M\equiv\GL_{r}(\F_{\mfp})\pmod \mfp$, and $M$ mod $\mfp^2$ contains a non-scalar matrix which is congruent to the identity modulo $\mfp$. Then $M=\GL_{r}(A_{\mfp})$.
	\end{prop}
    We will show that the image $M:=\rho_{\varphi,\mfp}(G_F)$, where $\varphi$ is a Drinfeld $A$-module over $F$ as in Theorem~\ref{mfp_adic_surjectivity_main_theorem_1}, satisfies the hypotheses of Proposition~\ref{Pink_R_cond_T_adic_sur}. In this direction, as a first step, we establish the following.
    \begin{prop}
    \label{det_of_mfp_adic_rep_sujective}
        Let $\varphi, \mfp$ be as in Theorem~\ref{mfp_adic_surjectivity_main_theorem_1}. Then the determinant map
        \begin{equation*}
            \det\rho_{\varphi,\mfp}:G_{F}\ra A_{\mfp}^{\times}
        \end{equation*}
        is surjective, i.e., $\det(M)=A_{\mfp}^{\times}$.
    \end{prop}
	The proof of Proposition~\ref{det_of_mfp_adic_rep_sujective} requires the following result from Hayes~\cite{Hay74}.
    \begin{prop}[\cite{Hay74}]
		\label{Hayes}
		For every non-zero ideal $\mathfrak{a}$ of $A$, the representation
        \begin{equation*}
            \bar{\rho}_{C,\mathfrak{a}}:G_{F}\ra \Aut(C[\mathfrak{a}])\cong (A/\mathfrak{a})^{\times}
        \end{equation*}
		is surjective, and the adelic Galois representation of the Carlitz module is surjective. Moreover, for $\mfp\in\Omega_A$ such that $\mfp\nmid\mfa$, we have $\bar{\rho}_{C,\mathfrak{a}}(\Frob_{\mfp})\equiv \mfp \pmod{\mathfrak{a}}$. 
	\end{prop}
    \begin{proof}[Proof of Proposition~\ref{det_of_mfp_adic_rep_sujective}]
       	By~\cite[Theorem 3.7.1(3)]{Pap23}, we get $\det(M) =\rho_{\psi,\mfp}(G_{F})$, where $\psi_{T}=T+g_2^{q-1}\tau$. Since $g_2\psi_{T}=C_{T}g_2$,  $\psi$ and $C$ are isomorphic over $F$. Hence, by Proposition~\ref{Hayes}, we have $\det(M)=\rho_{\psi,\mfp}(G_{F})=\rho_{C,\mfp}(G_{F})=A_\mfp^{\times}$.
    \end{proof}
    To verify that $M$ satisfies the remaining hypotheses of Proposition~\ref{Pink_R_cond_T_adic_sur}, we first analyze the subgroup $\bar{\rho}_{\varphi, \mfp}(I_{\mfl})$ of $\overline{M}:=\bar{\rho}_{\varphi,\mfp}(G_F)\subseteq\GL_{2}(\F_\mfp)$, where $I_{\mfl}$ denotes the inertia subgroup of $G_F$ at $\mfl$. To achieve the object, we now reprove a result of Zywina \cite[Proposition 4.1]{Zyw11} under appropriate hypotheses.
    
    \begin{prop}
    \label{corrected_zywina_prop}
    Let $\tilde{\mfp}\in\Omega_A$, and let $\varphi$ be a generic Drinfeld $A$-module of rank $2$ over $F_{\tilde{\mfp}}$ defined by $\varphi_T=T+\tilde{g}_1\tau+\tilde{g}_2\tau^2\in A_{\tilde{\mfp}}\{\tau\}$, and suppose that $\varphi$ has stable reduction of rank $1$ at $\tilde{\mfp}$. Let $\mfa \subseteq A$ be a non-zero proper ideal with $(\tilde{\mfp},\mfa)=1$.
    \begin{enumerate}
        \item[(i)] There exists a basis of $\varphi[\mfa]$ over $A/\mfa$ such that
        \begin{equation*}
            \bar{\rho}_{\varphi,\mfa}(I_{\tilde{\mfp}})\subseteq\left\{\begin{pmatrix}
                1 & b\\
                0 & c
            \end{pmatrix}:
            b\in A/\mfa,\;
            c\in\F_q^\times
            \right\}.
        \end{equation*}
        \item[(ii)] Let $e_\varphi$ be the order of $\frac{1}{(q-1)q^{\deg_{T}(\mfa)}}\nu_{\tilde{\mfp}}\left(\frac{\tilde{g}_1^{q+1}}{\tilde{g}_2}\right)+\Z$ in $\Q/\Z$. Then $\left|\bar{\rho}_{\varphi,\mfa}(I_{\tilde{\mfp}})\right|
        \ge e_\varphi.$
    \end{enumerate}
    \end{prop}
    \begin{proof}
        Since the generic Drinfeld $A$-module $\varphi$ of rank $2$ over $F_{\tilde{\mfp}}$ has stable reduction of rank $1$ at $\tilde{\mfp}$, by Theorem~\cite[Theorem 6.2.11]{Pap23}, let $(\psi,\Lambda)$ be the corresponding Tate datum, where $\psi$ is a generic Drinfeld $A$-module of rank $1$ over $F_{\tilde{\mfp}}$ with $\psi_T\in A_{\tilde{\mfp}}\{\tau\}$ has good reduction at $\tilde{\mfp}$ and $\Lambda$ is a $\psi$-lattice of rank $1$. Now, there exists an $A[G_{F_{\tilde{\mfp}}}]$-module isomorphism
        \begin{equation*}
            e_{\Lambda}: \psi_{\mfa}^{-1}(\Lambda)/\Lambda \xrightarrow{\sim} \varphi[\mfa]\, \text{ defined by }\, z+\Lambda \mapsto e_{\Lambda}(z)
        \end{equation*}
        where $e_{\Lambda}(x)\in A_{\tilde{\mfp}}\langle\!\langle x\rangle\!\rangle$ is an entire $\F_{q}$-linear function (cf. \cite[\S 6.2]{Pap23}). Hence, it is enough to find a basis of $ \psi_{\mfa}^{-1}(\Lambda)/\Lambda$ such that the action of $I_{\tilde{\mfp}}$ on $ \psi_{\mfa}^{-1}(\Lambda)/\Lambda$ is of the form $ \left\{ \psmat{1}{b}{0}{c} :b\in A/\mfa,\, c\in \F_{q}^{\times}\right\}$.
		
		The Drinfeld $A$-module $\psi$ over $F_{\tilde{\mfp}}$ is of rank $1$ and generic, by~\cite[Corollary 3.5.3]{Pap23}, there is an $A/\mfa$-basis $\{w\}$ of $\psi[\mfa]$. Since $\psi$ 
		has good reduction at $\tilde{\mfp}$ and $(\tilde{\mfp},\mfa)=1$, the Galois representation $\bar{\rho}_{\psi,\mfa}: G_{F_{\tilde{\mfp}}}\ra\Aut(\psi[\mfa])$ is unramified at $\tilde{\mfp}$. In particular, $\sigma(w)=w$ for all $\sigma\in I_{\tilde{\mfp}}$.
		%Note that since $\psi[\mfl]\cong \F_{\mfl}^2$ as an $\F_{\mfl}$-module  so we get such $\F_{\mfl}$-basis of $\psi[\mfl]$.

        Since $\Lambda$ is a $\psi$-lattice of rank $1$, we may fix a generator $\lambda\in(F_{\tilde{\mfp}}^{\sep})^{\times}$ of $\Lambda$ with $\nu_{\tilde{\mfp}}(\lambda)<0$, hence $\Lambda=\{\psi_{a}(\lambda):a\in A\}$ (cf. \cite[Example~6.2.2]{Pap23}). We can choose $z\in F_{\tilde{\mfp}}^{\sep}$ such that $\psi_{\mfa}(z)= \lambda$. Since $\Lambda$ is $G_{F_{\tilde{\mfp}}}$-stable, there is a character
		$\chi_{\Lambda}:G_{F_{\tilde{\mfp}}}\ra A^{\times}=\F_{q}^{\times}$ such that $\sigma(\lambda)=\chi_{\Lambda}(\sigma)\lambda$ for all $\sigma\in G_{F_{\tilde{\mfp}}}$. For any $\sigma\in I_{\tilde{\mfp}}$, we have
        \begin{equation*}
            \psi_{\mfa}(\sigma(z))=\sigma(\psi_{\mfa}(z))=\sigma(\lambda)=\chi_{\Lambda}(\sigma)\lambda=\chi_{\Lambda}(\sigma)\psi_{\mfa}(z)=\psi_{\mfa}(\chi_{\Lambda}(\sigma)z).
        \end{equation*}
		Thus, $\sigma(z)-\chi_{\Lambda}(\sigma)z\in \psi[{\mfa}]$, therefore there is an element $b_{\sigma}\in A/\mfa$ such that
		$\sigma(z)-\chi_{\Lambda}(\sigma)z=b_{\sigma}w$. Therefore, the action of $\sigma\in I_{\tilde{\mfp}}$ on $ \psi_{\mfa}^{-1}(\Lambda)/\Lambda$, with respect to the basis $\{w+\Lambda, z+\Lambda\}$, is of the form $\psmat{1}{b_{\sigma}}{0}{\chi_{\Lambda}(\sigma)}$. This proves part~(i).

        For part~(ii), recall that
        \begin{align}
        \label{image_of_inertia_unram}
            \bar\rho_{\varphi,\mfa}(I_{\tilde{\mfp}})\cong I_{\tilde{\mfp}}/\Gal(F_{\tilde{\mfp}}^{\sep}/F_{\tilde{\mfp}}^{\un}(\varphi[\mfa]))\cong\Gal\bigl(F_{\tilde{\mfp}}^{\un}(\varphi[\mfa])/F_{\tilde{\mfp}}^{\un}\bigr),
        \end{align}
        where $F_{\tilde{\mfp}}^{\un}$ is the maximal unramified extension of $F_{\tilde{\mfp}}$, i.e., the fixed field of $I_{\tilde{\mfp}}$. Now, from part~(i), we have $F_{\tilde{\mfp}}^{\un}(\varphi[\mfa])=F_{\tilde{\mfp}}^{\un}(e_{\Lambda}(w),e_{\Lambda}(z))$. 

         We first show that $F_{\tilde{\mfp}}^{\un}(\varphi[\mfa]) = F_{\tilde{\mfp}}^{\un}(z)$. Since $\psi[\mfa]$ is unramified, we get $w, e_\Lambda(w)$ belongs to $F_{\tilde{\mfp}}^{\un}$, because $\sigma(e_\Lambda(w))=e_\Lambda(\sigma w)=e_\Lambda(w)$ for all $\sigma\in I_{\tilde{\mfp}}$. Therefore, $F_{\tilde{\mfp}}^{\un}(\varphi[\mfa])=F_{\tilde{\mfp}}^{\un}(e_{\Lambda}(z))$. By definition of $e_{\Lambda}$, we get an inclusion $F_{\tilde{\mfp}}^{\un}(e_{\Lambda}(z))\subseteq F_{\tilde{\mfp}}^{\un}(z)$. For the reverse inclusion, for any $\sigma\in \Gal\bigl(F_{\tilde{\mfp}}^{\sep}/F_{\tilde{\mfp}}^{\un}(e_\Lambda(z))\bigr)$, $e_\Lambda(\sigma(z))=\sigma(e_\Lambda(z))=e_\Lambda(z),$ and hence $\sigma(z)-z\in\Lambda.$ Thus $\sigma(z)-z=\psi_b(\lambda)$ for some $b\in A$. 
         
         We now show that $b=0$. If $b\neq0$, then $\nu_{\tilde{\mfp}}\bigl(\psi_b(\lambda)\bigr)=q^{\deg_{T}(b)}\nu_{\tilde{\mfp}}(\lambda)$, since $\psi$ has rank $1$ and good reduction at $\tilde{\mfp}$ (cf. \cite[Example 6.2.2]{Pap23}). Recall from part~(i) that $\psi_{\mfa}(z)=\lambda$, we get $\nu_{\tilde{\mfp}}(\lambda)=q^{\deg_{T}(\mfa)}\nu_{\tilde{\mfp}}(z)$, again by~\textit{loc. cit.}
         Since $\psi_T\in A_{\tilde{\mfp}}\{\tau\}$ and $\nu_{\tilde{\mfp}}(\lambda)<0$, we get $\nu_{\tilde{\mfp}}(z)<0$. Hence, we obtain
         \begin{equation*}
             \nu_{\tilde{\mfp}}\bigl(\psi_b(\lambda)\bigr) =q^{\deg_{T}(b)+\deg_{T}(\mfa)}\nu_{\tilde{\mfp}}(z)<\nu_{\tilde{\mfp}}(z).
         \end{equation*}
        Consequently,
             $\nu_{\tilde{\mfp}}(\sigma(z))=\nu_{\tilde{\mfp}}\bigl(z+\psi_b(\lambda)\bigr)=\nu_{\tilde{\mfp}}\bigl(\psi_b(\lambda)\bigr)<\nu_{\tilde{\mfp}}(z),$
         which is a contradiction, since $I_{\tilde{\mfp}}$ preserves the valuation. Hence $b=0$, therefore $\sigma(z)=z$ for all $\sigma\in \Gal\bigl(F_{\tilde{\mfp}}^{\sep}/F_{\tilde{\mfp}}^{\un}(e_\Lambda(z))\bigr)$. This proves that $F_{\tilde{\mfp}}^{\un}(e_{\Lambda}(z))=F_{\tilde{\mfp}}^{\un}(z)$, i.e., $F_{\tilde{\mfp}}^{\un}(\varphi[\mfa])=F_{\tilde{\mfp}}^{\un}(z)$. 
         
         The ramification index $e$ of $F_{\tilde{\mfp}}^{\un}(z)/F_{\tilde{\mfp}}^{\un}$ is at least the order $e_\varphi$ of $\nu_{\tilde{\mfp}}(z)+\Z$ in $\Q/\Z$, since $\nu_{\tilde{\mfp}}(z)\in e^{-1}\Z$. By~\cite[Lemma 5.3]{Ros03}, we have $\nu_{\tilde{\mfp}}(\lambda)=\frac{1}{q-1}\nu_{\tilde{\mfp}}(\frac{\tilde{g}_1^{q+1}}{\tilde{g}_2})$ and thus
         \begin{equation*}
             \nu_{\tilde{\mfp}}(z)=q^{-\deg_{T}(\mfa)}\nu_{\tilde{\mfp}}(\lambda)=\frac{1}{(q-1)q^{\deg_{T}(\mfa)}}\nu_{\tilde{\mfp}}\left(\frac{\tilde{g}_1^{q+1}}{\tilde{g}_2}\right).
         \end{equation*}
         Thus, by \eqref{image_of_inertia_unram}, $\left|\bar\rho_{\varphi,\mfa}(I_{\tilde{\mfp}})\right|\ge e_\varphi.$ This completes the proof of the proposition.
    \end{proof}
    
    \begin{lem}
    \label{unipotent_contained_in_image_of_inertia}
		Let $\varphi, \mfl, \mfp$ be as in Theorem~\ref{mfp_adic_surjectivity_main_theorem_1}. Then there exists a basis of $\varphi[\mfp]$ such that 
        \begin{equation*}
            \bar{\rho}_{\varphi,\mfp}(I_\mfl)=\left\{ \begin{pmatrix} 1 & b_{\sigma,1} \\ 0 & 1 \end{pmatrix} \; \middle| \; b_{\sigma,1} \in \F_{\mfp}\right\}.
        \end{equation*}
        In particular, $|\F_{\mfp}|=|A/\mfp|=q^{\deg_{T}(\mfp)}$ divides $|\overline{M}|$.
	\end{lem}
    \begin{proof}
        The Drinfeld $A$-module $\varphi$ over $F$ has stable reduction of rank $1$ at $\mfl$, since $\mfl\mid g_2$ and $\mfl\nmid g_1$. By Proposition~\ref{corrected_zywina_prop}, there exists a basis of $\varphi[\mfp]$ such that 
        \begin{align}
        \label{inertia_contained_in_a_nice_group}
        \bar{\rho}_{\varphi,\mfp}(I_\mfl)\subseteq \left\{
			\psmat{1}{b_{\sigma,1}}{0}{b_{\sigma,2}} \; \middle| \; b_{\sigma,1} \in \F_{\mfp}, b_{\sigma,2}\in \F_{q}^{\times}\right\}
        \end{align}
        and $|\bar{\rho}_{\varphi,\mfp}(I_\mfl)|\geq e_{\varphi}$, where $e_{\varphi}$ is the order of $\frac{1}{(q-1)q^{\deg_{T}(\mfp)}}\nu_{\mfl}\Big(\frac{g_1^{q+1}}{-g_2^{q-1}}\Big)+\mathbb{Z}$ in $\mathbb{Q}/\mathbb{Z}$. Since $\nu_{\mfl}(g_1)=0$, and $p\nmid \nu_{\mfl}(g_2)$, we get $e_{\varphi}\geq q^{\deg_{T}(\mfp)}$, and hence
        $|\bar{\rho}_{\varphi,\mfp}(I_\mfl)|\geq q^{\deg_{T}(\mfp)}$. 

        By~\cite[Theorem 3.7.1(1)]{Pap23}, we have $\det\bar{\rho}_{\varphi,\mfp}(G_{F})=\bar{\rho}_{\psi,\mfp}(G_{F})$, where $\psi_{T}=T+g_2^{q-1}\tau$. Since $g_2\psi_{T}=C_{T}g_2$, $\psi$ and $C$ are isomorphic over $F$, and hence we have $\det\bar{\rho}_{\varphi,\mfp}(G_{F})=\bar{\rho}_{C,\mfp}(G_{F})$. Note that $(\mfl,\mfp)=1$. Since the Carlitz module $C$ has good reduction at $\mfl$, we get $\bar{\rho}_{C,\mfp}$ is unramified at $\mfl$, i.e., $\bar{\rho}_{C,\mfp}(I_{\mfl})=1$. Therefore, we get $\det\bar{\rho}_{\varphi,\mfp}(I_{\mfl})=1$. Now, the claim follows from~\eqref{inertia_contained_in_a_nice_group} together with the inequality $|\bar{\rho}_{\varphi,\mfp}(I_\mfl)|\geq q^{\deg_{T}(\mfp)}$.
    \end{proof}

    Note that the proof of Lemma~\ref{unipotent_contained_in_image_of_inertia} works for any non-zero ideal $\mfa\subseteq A$ with $(\mfl,\mfa)=1$, hence we have $\bar{\rho}_{\varphi,\mfa}(I_\mfl)=\left\{\psmat{1}{b_{\sigma,1}}{0}{1}\; \middle| \; b_{\sigma,1} \in A/\mfa\right\}.$ Therefore, by taking $\mfa=\mfp^2$, we get the following proposition: 

    \begin{prop}
    \label{non-scalar_mod_T_square_2}
        Let $\varphi, \mfp$ be as in Theorem~\ref{mfp_adic_surjectivity_main_theorem_1}. Then the image of $\bar{\rho}_{{\varphi,\mfp^2}}$, i.e., $M$ mod $\mfp^2$, contains a non-scalar matrix that is congruent to the identity modulo $\mfp$.
    \end{prop}
 
    The only remaining hypothesis to verify in Proposition~\ref{Pink_R_cond_T_adic_sur} is that the representation $\bar{\rho}_{\varphi,\mfp}$ is surjective, i.e., $M\equiv\GL_{2}(\F_{\mfp})\pmod{\mfp}$. By Proposition~\ref{det_of_mfp_adic_rep_sujective}, we get $\det(\overline{M})=\F_{\mfp}^{\times}$. The required claim follows if we can show that $\SL_{2}(\F_{\mfp})\subseteq \overline{M}$, which can be achieved by~\cite[Lemma A.1]{Zyw11}. 
    \begin{lem}
    \label{zyw_sl_{2}_F}
        Let $G$ be a subgroup of $\GL_2(\F)$, where $\F$ is a finite field. If $G$ contains a Sylow $p$-subgroup of $\GL_2(\F)$ and $G$ acts irreducibly on $\F^2$, then $G$ contains $\SL_{2}(\F)$. 
     \end{lem}
    By Lemma~\ref{unipotent_contained_in_image_of_inertia}, the group $\overline{M}$ contains the Sylow $p$-subgroup of $\GL_2(\F_q)$, consisting of all unipotent upper triangular matrices. By Lemma~\ref{zyw_sl_{2}_F}, either we have $\SL_2(\F_\mfp)\subseteq\overline{M}$ or the action of $\overline{M}$ on $\varphi[\mfp]\cong \F_{\mfp}^2$ is reducible, which cannot happen by the following proposition. 

    \begin{prop}
    \label{irreducible_rep_at_mfp}
        Let $\varphi, \mfp$ be as in Theorem~\ref{mfp_adic_surjectivity_main_theorem_1}. Then $\overline{M}$ acts irreducibly on $\varphi[\mfp]$, i.e., the $\F_{\mfp}[G_F]$-module $\varphi[\mfp]$ is irreducible.
    \end{prop}
    \begin{proof}
        On the contrary, suppose that the $\F_{\mfp}[G_F]$-module $\varphi[\mfp]$ is reducible. By assumption, the module $\varphi$ has a good reduction at $(T-c)\in \Omega_{A}\setminus\{\mfp\}$, therefore $\rho_{\varphi,\mfp}$ is unramified at $(T-c).$ So, the matrix $\rho_{\varphi,\mfp}(\Frob_{(T-c)})\in \GL_{2}(A_\mfp)$ is well-defined up to conjugation, where $\Frob_{(T-c)}$ denotes the Frobenius at $(T-c)$. Moreover, its characteristic polynomial $P_{\varphi,(T-c)}(x)=\det(xI_{2}-\rho_{\varphi,\mfp}(\Frob_{(T-c)}))=x^2+a_1x+a_0$ belongs to $A[x]$. By~\cite[Theorem 4.2.7 (2, 3)]{Pap23}, we have $a_1\in \F_{q}$ and
        \begin{equation*}
            a_{0}=-\Nr_{\F_{(T-c)/\F_q}}(-g_2(c)^{q-1})^{-1}\cdot(T-c)=(T-c).
        \end{equation*}
        Therefore, $P_{\varphi,(T-c)}(x)=x^2+a_{1}x+(T-c)\in A[x]$, where $a_{1}\in \F_{q}$. Since $P_{\varphi,(T-c)}(x)$ is also the characteristic polynomial of the Frobenius endomorphism of
		$\varphi\otimes\F_{(T-c)}$, the reduction of $\varphi_T$ modulo $(T-c)$, acting on $T_{\mfp}(\varphi\otimes\F_{(T-c)})$, we get
		\begin{equation}
        \label{ch_mod_l_eq_T}
			\tau^2+(\varphi\otimes\F_{(T-c)})_{a_{1}}\tau+(\varphi\otimes\F_{(T-c)})_{T-c}=0.
		\end{equation}
		Since $\varphi_{T}=T+g_{1}\tau-g_{2}^{q-1}\tau^2$, we have $(\varphi\otimes\F_{(T-c)})_{T}=c+g_{1}(c)\tau-g_{2}(c)^{q-1}\tau^2=c+g_{1}(c)\tau-\tau^2$ (as $g_{2}(c)\in \F_{q}^{\times}$). Therefore, $(\varphi\otimes\F_{(T-c)})_{T-c}=g_{1}(c)\tau-\tau^2$. Hence, by~\eqref{ch_mod_l_eq_T}, $a_{1}=-g_{1}(c)$. So, we have 
        \begin{equation*}
            P_{\varphi,(T-c)}(x)=x^2-g_{1}(c)x+(T-c)\in A[x].
        \end{equation*} 
        
        Since the characteristic polynomial of $\bar{\rho}_{\varphi,\mfp}(\Frob_{(T-c)})$ is congruent to $P_{\varphi,(T-c)}(x)$ modulo $\mfp$, the polynomial $x^2-g_{1}(c)x+(T-c)$ is reducible over $\F_{\mfp}$ because $\varphi[\mfp]$ is reducible. This contradicts our assumption, and hence the claim follows.
    \end{proof}

    Since $M:=\rho_{\varphi,\mfp}(G_F)$ satisfies the hypotheses of Proposition~\ref{Pink_R_cond_T_adic_sur}, the proof of Theorem~\ref{mfp_adic_surjectivity_main_theorem_1} is complete.
    
    The following corollary is a consequence of Theorem~\ref{mfp_adic_surjectivity_main_theorem_1}.
    \begin{cor}
    \label{cor_padic_inverse_galois}
        Let $\varphi, \mfp$ be as in Theorem~\ref{mfp_adic_surjectivity_main_theorem_1}, let $F_{\mfp^\infty}:= \bigcup_{n\ge1}F(\varphi[\mfp^n])$. Then $\Gal(F_{\mfp^\infty}/F)\simeq \GL_2(A_{\mfp}).$ In particular, we also have $\Gal(F(\varphi[\mfp])/F)\simeq \GL_2(A/\mfp)$. 
    \end{cor}
    
     Therefore, the $\mfp$-adic surjectivity provides a realisation of the finite (resp., infinite profinite) group $\GL_2(A/\mfp)$ (resp., $\GL_2(A_{\mfp})$) as a Galois group over $F$. We conclude this section with several remarks that highlight comparisons between our results and those in~\cite{Ray24}.

    \subsection{Comparison with~\cite{Ray24}}
    \begin{itemize}
        \item In this article, for an arbitrary $\mfp \in \Omega_A$, we provide easily verifiable sufficient conditions on the coefficients of $\varphi_T$ that ensure the $\mfp$-adic surjectivity of the associated Galois representation. In~\cite{Ray24}, a variant of our results was shown for the prime $\mfp = (T)$ with a hypothesis different from ours. 
        \item The approach in~\cite{Ray24} relies on working with Drinfeld $A$-modules of rank $2$ over $F$ of reduction height $1$ at $(T)$ (cf.~\cite[Proposition 2.3]{Ray24}). For $\mfp \in \Omega_A$ with $\deg_T(\mfp) \gg 1$, a similar approach forces us to have 
        stronger divisibility conditions on the coefficients of $\varphi_{\mfp}$. To allow
        $\mfp \in \Omega_A$ with $\deg_T(\mfp) \gg 1$, we consider the Drinfeld $A$-modules of rank $2$ over $F$ defined by $\varphi_T = T + g_1\tau - g_2^{\,q-1}\tau^2 \in A\{\tau\}$, thereby allowing the reduction height at most $2$ at $(T)$, and weaker congruence conditions, compared to~\cite[Theorem 3.1]{Ray24}, would imply $\mfp$-adic surjectivity.
        
        \item Our method relies on weaker valuation conditions. To prove Lemma~\ref{unipotent_contained_in_image_of_inertia} and Proposition~\ref{non-scalar_mod_T_square_2}, for any $\mfp \in \Omega_A$, we allow an auxiliary $\mfl \in \Omega_A \setminus \{\mfp,(T-c)\}$ satisfying $\nu_{\mfl}(g_1)=0$ and $p \nmid \nu_{\mfl}(g_2)$. In contrast, in~\cite{Ray24}, for the case $\mfp = (T)$ to achieve the same (cf.~\cite[Theorem 3.1 and Proposition 3.5]{Ray24}), Ray used a degree-$1$ prime $(T-a_2) \in \Omega_A$ with stronger valuation conditions, namely $\nu_{(T-a_2)}(g_1)=0$ and $\nu_{(T-a_2)}(g_2)=1$.
    \end{itemize}

	\section{Vertical $\mfp$-adic surjectivity of Drinfeld modules}

    In this section, we show that for a fixed Drinfeld $A$-module $\varphi$ of rank $2$ over $F$ with $\varphi_{T}\in A\{\tau\}$, whose coefficients satisfy certain congruence and valuation conditions, the $\mfp$-adic Galois representation $\rho_{\varphi,\mfp}$ is surjective for all $\mfp \in \Omega_A$.

    Let $\Lambda_{A}$ be the set of pairs $(\mfl,r_{1}) \in \Omega_A \times \F_q$ such that the polynomial $x^2-r_1 x+(T-c)$ is irreducible over $\F_\mfl$ for some $c\in \F_q$. Clearly, we have $\mfl\neq (T-c)$.

    \begin{lem}\label{Lambda_A_is_non_empty}
        The set $\Lambda_{A}$ is non-empty.
        \end{lem}
        \begin{proof}
        For any $\mfl=(T-d)\in \Omega_{A}$ and $r_1\in\F_{q}^{\times}$, we show that there exist $c\in \F_{q}\setminus\{d\}$ such that $f_{c}(x):=x^2-r_{1}x+(T-c)$ is irreducible over $\F_{(T-d)}$, i.e., $\bar{f}_{c}(x):=x^2-r_{1}x+(d-c)\in \F_{q}[x]$ is irreducible.
        
         Assume that for all $c\in \F_{q}\setminus\{d\}$, the polynomial $\bar{f}_{c}(x)=x^2-r_{1}x+(d-c)\in \F_{q}[x]$ is reducible, i.e., it has roots in $\F^\times_{q}$, since $d\neq c$. Since the roots of $\bar{f}_{c_1}(x)$ and $\bar{f}_{c_2}(x)$, for $c_1 \neq c_2$, are distinct, every element in $\F^\times_{q}$ must be a root of some $\bar{f}_{c}(x)$ for $c\in \F_{q}\setminus\{d\}$. This is a contradiction, since $r_1$ is not a root of $\bar{f}_{c}(x)$ for any $c\in \F_q\setminus\{d\}$. Therefore, $(\mfl, r_1) \in \Lambda_{A}$ for all $\mfl\in\Omega_{A}$ with $\deg_{T}(\mfl)=1$ and $r_1\in \F_{q}^{\times}$.
      \end{proof}

    Now, we state the main theorem of this section.
    \begin{thm}
    \label{for_all_mfp_adic}
        Let $q=p^e\geq 5$ be an odd prime power and $(\mfl,r_{1})\in \Lambda_{A}$. Let $\varphi$ be a Drinfeld $A$-module of rank $2$ over $F$ defined by
        \begin{equation*}
            \varphi_{T}=T+g_{1}\tau-\mfl^{q-1}\tau^2\in A\{\tau\},
        \end{equation*}
        with $g_1\equiv r_1\pmod{T^q-T}$ and $\nu_{\mfl}(g_1)=0$. Then, the $\mfp$-adic Galois representation 
        \begin{equation*}
            \rho_{\varphi,\mfp}:G_{F}\longrightarrow \GL_2(A_{\mfp})
        \end{equation*}
		is surjective for all $\mfp\in \Omega_{A}$.
    \end{thm}
    We remark that the Drinfeld $A$-modules in Theorem~\ref{for_all_mfp_adic} not only generalise the example in~\cite{Zyw11} but also produce complementary examples considered in~\cite{Zyw25}.

    \begin{prop}
    \label{infinitely_many_g_1}
        Let $(\mfl,r_1)\in \Lambda_A$ with either $\deg_T(\mfl)>1$ or $r_1 \in \F_q^\times$. Then there exist infinitely many $g \in A$ such that
        \begin{equation*}
            g\equiv r_1 \pmod{T^q - T}\quad \text{and} \quad \nu_{\mfl}(g)=0.
        \end{equation*}
    \end{prop}
    \begin{proof}
        For $(\mfl,0) \in \Lambda_{A}$ with $\deg_{T}(\mfl)>1$, one can take $g$ to be any positive integral power of $(T^q-T)$. For $r_1 \in \F_q^\times$, define an infinite set
        \begin{equation*}
            \mathcal{S}_{r_1}:=\{g\in A:g\equiv r_1\pmod{T^q-T}\}.
        \end{equation*}
        We now show that there exist infinitely many $g \in \mathcal{S}_{r_1}$ with
        $\nu_{\mfl}(g)=0$. For $\deg_{T}(\mfl)=1$, every $g \in \mathcal{S}_{r_1}$ satisfies
        $\nu_{\mfl}(g)=0$ as $r_{1} \in \F_q^\times$.
        
        For $\deg_{T}(\mfl)>1$, assume to the contrary that there are only finitely many $g\in \mathcal{S}_{r_1}$ with $\nu_{\mfl}(g)=0$. For $n\geq 1$, define $h^{(n)}:=r_1+(T^q-T)T^n \in \mathcal{S}_{r_1}$. Now, if $\mfl\mid h ^{(n+1)}$ and $\mfl\mid h^{(n)}$, then $\mfl\mid (h^{(n+1)}-h^{(n)})=(T^{q}-T)T^{n}(T-1)$, which cannot be true since $\deg_{T}(\mfl)>1$. This implies that there are infinitely many $n \in \N$ with $\nu_\mfl (h^{(n)}) =0$, which is a contradiction. Hence, this completes the proof.
    \end{proof}
    Therefore, by Lemma~\ref{Lambda_A_is_non_empty} and Proposition~\ref{infinitely_many_g_1}, there are infinitely many examples satisfying Theorem~\ref{for_all_mfp_adic}.

   \subsection{Notations and relevant results:}

    Let $\varphi$ be a Drinfeld $A$-module of rank $2$ over $F$ defined by $\varphi_{T}=T+g_{1}\tau+g_{2}\tau^2 \in A\{\tau\}$ with good reduction at $\mfp \in \Omega_A$, and let $H_{\mfp}$ be the height of $\varphi$ modulo $\mfp$.
   
    Set $n_\mfp :=q^{H_{\mfp}\deg_{T}(\mfp)}$. The Newton polygon of $\frac{\varphi_{\mfp}(x)}{x}$ implies that $\varphi[\mfp]$ contains $n_{\mfp}-1$ elements with valuation $\frac{1}{n_{\mfp}-1}$ and the remaining elements have valuation $0$. Let $V$ be the set of elements with valuation $\frac{1}{n_{\mfp}-1}$. 
    
    We define $\varphi[\mfp]^{0} :=\{\alpha\in\varphi[\mfp]:|\alpha|_{\mfp}<1\}=V\cup\{0\},$ which is an $A$-submodule of $\varphi[\mfp]$. This is because if $|\alpha|_{\mfp}<1$ then $|\varphi_a(\alpha)|_{\mfp} < 1$, for any $a \in A$, as $\varphi_a(x) \in A[x]$. Note that $\varphi[\mfp]^{0}\cong (\F_\mfp)^{H_{\mfp}}$. Moreover, the submodule $\varphi[\mfp]^{0}$ is $G_{F_{\mfp}}$-invariant, since $G_{F_{\mfp}}$ preserves the absolute value on $F_{\mfp}^{\sep}$. The commutativity of these actions gives rise to a Galois representation
    \begin{equation*}
        \bar{\rho}^{0}_{\varphi,\mfp}:G_{F_{\mfp}} \rightarrow \Aut_{A}(\varphi[\mfp]^{0}).
    \end{equation*}
	In particular, $\varphi[\mfp]^{0}$ is an $\F_{\mfp}[G_{F_{\mfp}}]$-module. 
    
    Let $\bar{\varphi}$ be the reduction of $\varphi$ at $\mfp$ and $\bar{\varphi}[\mfp]$ be its $\mfp$-torsion. Since $\varphi$ has good reduction at $\mfp$, we get
	$\varphi[\mfp]^{\text{\'{e}t}} = \bar{\varphi}[\mfp] \cong(\F_\mfp)^{2-H_{\mfp}}$. We have the following short exact sequence of $\F_{\mfp}[I_{\mfp}]$-modules (cf.~\cite[Section 6.3]{Pap23})
    \begin{equation*}
        0\lra \varphi[\mfp]^{0}\lra \varphi[\mfp]\lra \bar{\varphi}[\mfp]\lra 0.
    \end{equation*}
	
	Let $F_{\mfp}(V)$ be the fixed field of the kernel of $\bar{\rho}^{0}_{\varphi,\mfp}$, which is tamely ramified over $F_{\mfp}$ with ramification index $n_\mfp-1$. Moreover, $F_{\mfp}(V)$ is the splitting field of $f(x)=x^{n_\mfp}-\pi_{\mfp}x$ over $F_{\mfp}$, where $\pi_{\mfp}$ is a uniformizer of $F_{\mfp}$. The Galois group $\Gal(F_{\mfp}(V)/F_{\mfp})$ is isomorphic to $\F_{\mfp}^{(H_{\mfp})}$ by the fundamental character $\zeta_{n_{\mfp}}$, which is defined as follows:
    \begin{equation*}
        \zeta_{n_{\mfp}} : \operatorname{Gal}(F_{\mfp}(V)/F_{\mfp}) \to \F_{\mfp}^{(H_{\mfp})}, \quad \sigma \mapsto \sigma(\lambda)/\lambda,
    \end{equation*}
	where $\lambda$ is a root of $f(x)$, where $\F_{\mfp}^{(H_\mfp)}$ be the extension of $\F_{\mfp}$ with $n_\mfp$ elements.
	Now, we recall a proposition of Pink and R\"utsche (cf.~\cite[Proposition 2.7]{PR09}).
	\begin{prop}
	\label{PR_09}
		\begin{enumerate}
			\item The inertia group $I_{\mfp}$ acts trivially on $\bar{\varphi}[\mfp]$.
            \item The action of the wild inertia group at $\mfp$ on $\varphi[\mfp]^{0}$ is trivial.
			
			\item The $\F_{\mfp}$-vector space $\varphi[\mfp]^{0}$ extends uniquely to a one dimensional $\F_{\mfp}^{(H_{\mfp})}$-vector space structure such that the action of $I_{\mfp}$ on $\varphi[\mfp]^{0}$ is given by the fundamental character $\zeta_{n_\mfp}$.
			
		\end{enumerate}
	\end{prop}
	\begin{cor}[\cite{Che22}, Corollary 5.2] \label{ch_PR09}
		$\varphi[\mfp]^{0}$ is an irreducible $\F_{\mfp}[I_{\mfp}]$-module.
	\end{cor}
	
    \subsection{Proof of Theorem~\ref{for_all_mfp_adic}}
    
	As before, the proof of Theorem~\ref{for_all_mfp_adic} relies on Proposition~\ref{Pink_R_cond_T_adic_sur}, and hence it is enough to show that $M_{\mfp}:=\rho_{\varphi,\mfp}(G_F)$ satisfies the hypotheses of Proposition~\ref{Pink_R_cond_T_adic_sur} for all $\mfp\in\Omega_{A}$. First, arguing as in Proposition~\ref{det_of_mfp_adic_rep_sujective}, we get
    \begin{prop}
    \label{det_of_mfp_adic_rep_sujective_for_all}
        Let $\varphi$ be as in Theorem~\ref{for_all_mfp_adic}. Then the determinant map
        \begin{equation*}
            \det\rho_{\varphi,\mfp}:G_{F}\ra A_{\mfp}^{\times}
        \end{equation*}
        is surjective, i.e., $\det(M_{\mfp})=A_{\mfp}^{\times}$ for all $\mfp\in\Omega_{A}$.
    \end{prop}
    For all $\mfp\in\Omega_{A}$, to verify that $M_{\mfp}$ satisfies the remaining hypotheses of Proposition~\ref{Pink_R_cond_T_adic_sur}, we first analyze the subgroup $\bar{\rho}_{\varphi, \mfp}(I_{\mfl})$ of $\overline{M}_{\mfp}:=\bar{\rho}_{\varphi,\mfp}(G_F)\subseteq\GL_{2}(\F_\mfp)$.
    
    \begin{lem}
    \label{unipotent_contained_in_image_of_inertia_for_all_mfp}
		Let $\varphi$ be as in Theorem~\ref{for_all_mfp_adic}. Then, for all $\mfp\in\Omega_{A}$, we have $|\F_\mfp|=q^{\deg_{T}(\mfp)}$ divides $|\overline{M}_{\mfp}|$. Moreover, for $\mfp\neq \mfl$, there exists a basis of $\varphi[\mfp]$ such that 
        \begin{equation*}
            \bar{\rho}_{\varphi,\mfp}(I_\mfl)=\left\{ \begin{pmatrix} 1 & b_{\sigma,1} \\ 0 & 1 \end{pmatrix} \; \middle| \; b_{\sigma,1} \in \F_{\mfp}\right\}.
        \end{equation*}
        In particular, for all $\mfp\in\Omega_{A}$, $\overline{M}_{\mfp}$ contains a Sylow $p$-subgroup of $\GL_{2}(\F_\mfp)$. 
    \end{lem}
    \begin{proof}
        We first assume $\mfp=\mfl$. Recall that \begin{equation}
        \label{image_of_decomposition_at_mfl_via_mod_mfl_rep}
			\bar{\rho}_{\varphi,\mfl}(G_{F_{\mfl}}) \cong G_{F_{\mfl}}/\Gal\left(F^{\sep}_{\mfl}/F_{\mfl}(\varphi[\mfl])\right)\cong \Gal\left(F_{\mfl}(\varphi[\mfl])/F_{\mfl}\right),
		\end{equation}
		where $G_{F_{\mfl}}$ is the decomposition subgroup of $G_{F}$ at $\mfl \in \Omega_A$ and $F_{\mfl}(\varphi[\mfl])$ is the smallest extension of $F_{\mfl}$ such that $\Gal(F_{\mfl}^{\sep}/F_{\mfl}(\varphi[\mfl]))$ acts trivially on $\varphi[\mfl]$. 
        
        For $0 \neq a \in A$, we let $\nu_{\mfl}^{a}$ denote the unique valuation on $F_{\mfl} (\varphi[a])$ extending $\nu_{\mfl}$. By definition, we have $\nu_{\mfl}^{a}(F_{\mfl}\left(\varphi[a])^{\times}\right)=\dfrac{1}{e[F_{\mfl}(\varphi[a]):F_{\mfl}]}\Z$, where $e[F_{\mfl}(\varphi[a]): F_{\mfl}]$ denotes the ramification index of $F_{\mfl}(\varphi[a])/F_{\mfl}$.
				
		Let $l= \deg_{T}(\mfl)$ and hence $\deg_{\tau}(\varphi_{\mfl})=2l$. Since $\mfl\nmid g_1$, we get $\bar{\varphi}_{T}=\bar{T}+\bar{g}_{1}\tau$ be a Drinfeld $A$-module of rank $1$ over $\F_{\mfl}$, where the bar denotes the reduction modulo $\mfl$. Now, $\deg_{\tau}(\bar{\varphi}_{\mfl})=l$. Hence, we have 
        \begin{align}
            \nu_{\mfl}(\text{coefficient of }\tau^{l}\text{ in }\varphi_{\mfl})&= 0, \label{valuation_of_coeff_of_varphi_mfl_at_tau_cap_l} \\
            \nu_{\mfl}(\text{coefficient of }\tau^{k}\text{ in }\varphi_{\mfl})&\geq 1 \quad \text{for all } l<k\leq 2l.\label{valuation_of_coeff_of_varphi_mfl_at_tau_cap_k}
        \end{align}
        By Corollary~\ref{expression_of_varphi_a_with_explicit_coefficients}, the coefficient of $\tau^{l+1}$ in $\varphi_{\mfl}$ is
        \begin{equation*}
            \sum_{j=\big\lceil \tfrac{l+1}{2} \big\rceil}^{l}a_{j}\Big(\sum_{\substack{0\leq i_1,i_2,\ldots,i_j\leq 2 \\ i_1+i_2+\cdots+i_j=l+1}}^{}g_{i_{1}}g_{i_{2}}^{q^{i_{1}}}g_{i_{3}}^{q^{i_1+i_2}}\cdots g_{i_{j}}^{q^{i_1+i_2+\cdots+i_{j-1}}}\Big),
        \end{equation*}
        where $g_2=-\mfl^{q-1}$ and $g_0=T$.

        Observe that among all tuples $(i_1,i_2,\ldots,i_j)$, the minimal valuation $q-1$ occurs precisely for the tuple $(2,1,1,\ldots,1)$, i.e., $i_1=2$ and the remaining $(l-1)$ entries are $1$. In this case, the factor $g_2$ appears without any Frobenius twist and contributes valuation $q-1$. Any other placement of the entry $2$ increases the valuation (at least $q(q-1)$) due to Frobenius twisting. Moreover, all other tuples corresponding to partitions of $l+1$ contain at least two entries equal to $2$, and hence contribute terms of valuation at least $(1+q^2)(q-1)$. Therefore, the minimal valuation $q-1$ occurs uniquely, hence we get $\nu_{\mfl}(\text{coefficient of}\ \tau^{l+1} \ \text{in}\ \varphi_{\mfl})=q-1$. We now look at Newton's polygon of $\varphi_{\mfl}(x)/x\in A[x]$. 
        \begin{itemize}
            \item Since all coefficients of $\varphi_{\mfl}(x)/x$ have non-negative valuation at $\mfl$, by~\eqref{valuation_of_coeff_of_varphi_mfl_at_tau_cap_l} and \eqref{valuation_of_coeff_of_varphi_mfl_at_tau_cap_k}, there must be a line segment starting at $(q^l-1,0)$.
            \item For each $1<i\leq l$, arguing as above for the coefficient of $\tau^{l+i}$, we see that the point $(q^{l+i}-1,\nu_{\mfl}(\text{coefficient of}\ \tau^{l+i} \ \text{in}\ \varphi_{\mfl}))$ lies above the line joining $(q^l-1,0)$ and $(q^{l+1}-1,q-1)$. Therefore, the Newton polygon of $\varphi_{\mfl}(x)/x$ contains a line segment joining $(q^l-1,0)$ and $(q^{l+1}-1,q-1)$.
        \end{itemize}
        So, $\varphi_{\mfl}(x)/x$ has $q^{l+1}-q^{l}$ many roots $\alpha \in F_{\mfl}(\varphi[\mfl])$ with $\nu_{\mfl}^{\mfl}(\alpha)=-\frac{(q-1)}{q^{l}(q-1)}=-\frac{1}{q^l}$, and hence $q^{l}=|A/\mfl|$ divides $e[F_{\mfl}(\varphi[\mfl]): F_{\mfl}]$. Since $e[F_{\mfl}(\varphi[\mfl]): F_{\mfl}]$ divides $|\Gal(F_{\mfl}(\varphi[\mfl])/F_{\mfl})|$, $q^{l}$ divides $|\Gal(F_{\mfl}(\varphi[\mfl])/F_{\mfl})|=|\bar{\rho}_{\varphi,\mfl}(G_{F_{\mfl}})|$. In particular, we have $q^{l}$ divides $|\overline{M}_{\mfl}|$.
        
        Now, for $\mfp \neq \mfl$, arguing as in the proof of Lemma~\ref{unipotent_contained_in_image_of_inertia}, the claim follows.
    \end{proof}

    \begin{prop}
    \label{non-scalar_mod_mfp_square_for_all_mfp}
        Let $\varphi$ be as in Theorem~\ref{for_all_mfp_adic}. For all $\mfp\in\Omega_{A}$, the image $\bar{\rho}_{{\varphi,\mfp^2}}(G_{F})$, i.e., $M_{\mfp}\pmod {\mfp^2}$, contains a non-scalar matrix which is congruent to the identity modulo $\mfp$.
    \end{prop}
    \begin{proof}
        We first assume $\mfp=\mfl$ and $l = \deg_{T}(\mfl)$. In this case, arguing as in the proof of Lemma~\ref{unipotent_contained_in_image_of_inertia_for_all_mfp}, we get $\varphi_{\mfl^2}(x)/x$ has $q^{2l+1}-q^{2l}$ many roots in $F_{\mfl}(\varphi[\mfl^2])$ with valuation $-\frac{1}{q^{2l}}$ and $|A/\mfl^2|=q^{2l}$ divides $|\bar{\rho}_{\varphi,\mfl^2}(G_{F_{\mfl}})|$.
        Now, consider the surjective group homomorphism
        \begin{equation*}
            \bar{\rho}_{\varphi, \mfl^2}(G_F) \xrightarrow{\text{mod } \mfl} \bar{\rho}_{\varphi, \mfl}(G_F) \cong \GL_{2}(\F_{\mfl})
        \end{equation*}
		with kernel $S$. Note that $S \cap \bar{\rho}_{\varphi,\mfl^2}(G_{F_{\mfl}})$ is non-trivial, because $q^{2l}$ divides $|\bar{\rho}_{\varphi,\mfl^2}(G_{F_{\mfl}})|$ but not $|\GL_{2}(\F_{\mfl})|$.
        
        We now show that at least one non-trivial element in $S \cap \bar{\rho}_{\varphi,\mfl^2}(G_{F_{\mfl}})\subseteq \GL_2(A/{\mfl^2})$ cannot belong to $Z(\GL_2(A/{\mfl^2}))$. Suppose, for contradiction, that this is not the case. Then,
        for all $\sigma \in G_{F_{\mfl}} $ such that $\bar{\rho}_{\varphi,\mfl^2}(\sigma) \in S$ is a non-trivial scalar matrix $(1+a_\sigma\mfl)I_2$ for some $a_\sigma=x_0+x_1T+\cdots+x_{a}T^a\in A\setminus\{0\}$ with $\deg_{T}(a_\sigma)=a<l$. 
%        in the intersection acts on $\phi[\mfl^2]$ by ${1+a_\sigma\mfl}$, i.e.,

    Let $\alpha$ be a root of $\varphi_{\mfl^2}(x)/x$ with $\nu_{\mfl}^{\mfl^2}(\alpha)= -\frac{1}{q^{2l}}$. We compute
	\begin{align}
    \label{valuation_of_sigma_of_alpha}
		\nu_{\mfl}^{\mfl^2}(\sigma(\alpha))&=\nu_{\mfl}^{\mfl^2}\big(\varphi_{1+a_\sigma \mfl}(\alpha)\big)\notag
		=\nu_{\mfl}^{\mfl^2}\big(\alpha+(\varphi_{a_\sigma}\varphi_{\mfl})(\alpha)\big)\notag \\
		&=\nu_{\mfl}^{\mfl^2}\Big(\alpha+\big((x_0+x_1\varphi_{T}+\cdots+x_{a}\varphi_{T^a})\varphi_{\mfl}\big)(\alpha)\Big)\notag\\
		&=\nu_{\mfl}^{\mfl^2}\Big(\alpha+\big(x_0\varphi_{\mfl}+x_1\varphi_{T}\varphi_{\mfl}+\cdots+x_{a}\varphi_{T^a}\varphi_{\mfl}\big)(\alpha)\Big)\notag\\
		&=\nu_{\mfl}^{\mfl^2}\Big(\alpha+x_0\varphi_{\mfl}(\alpha)+x_1(\varphi_{T}\varphi_{\mfl})(\alpha)+\cdots+x_{a}(\varphi_{T^a}\varphi_{\mfl})(\alpha)\Big)\notag\\
		&\leq \min\Big\{\nu_{\mfl}^{\mfl^2}(\alpha),\nu_{\mfl}^{\mfl^2}\big(x_0\varphi_{\mfl}(\alpha)\big),\ldots,\nu_{\mfl}^{\mfl^2}\big(x_{a}(\varphi_{T^a}\varphi_{\mfl})(\alpha)\big)\Big\}
	\end{align}
    By~\eqref{valuation_of_coeff_of_varphi_mfl_at_tau_cap_l}, we get $\nu_{\mfl}^{\mfl^2}(\text{coefficient of}\ \tau^{l} \ \text{in}\ \varphi_{\mfl}) = 0$, and hence
	\begin{align}
    \label{valuation_of_x_0_varphi_mfl_of_alpha}
		\nu_{\mfl}^{\mfl^2}\big(x_0\varphi_{\mfl}(\alpha)\big)&\leq \nu_{\mfl}^{\mfl^2}\big(x_0(\text{coefficient of}\ \tau^{l} \ \text{in}\ \varphi_{\mfl})(\alpha)^{q^{l}}\big)\notag\\
		&=\nu_{\mfl}^{\mfl^2}(x_0)+\nu_{\mfl}^{\mfl^2}(\text{coefficient of}\ \tau^{l} \ \text{in}\ \varphi_{\mfl})+\nu_{\mfl}^{\mfl^2}((\alpha)^{q^{l}})\notag\\
		&=q^{l}\nu_{\mfl}^{\mfl^2}(\alpha)=- \frac{1}{q^{l}}<-\frac{1}{q^{2l}} \quad \text{as } l\geq 1. 
	\end{align}
    By \eqref{valuation_of_sigma_of_alpha} and \eqref{valuation_of_x_0_varphi_mfl_of_alpha}, we get $\nu_{\mfl}^{\mfl^2}(\sigma(\alpha))<-\frac{1}{q^{2l}}$, which is a contradiction, since $\sigma \in G_{F_{\mfl}}$ preserves valuation. Therefore, $M_{\mfl}\pmod{\mfl^2}$ contains a non-scalar matrix which is congruent to the identity modulo $\mfl$.
        
    For $\mfp \neq \mfl$, arguing as in the proof of Proposition~\ref{non-scalar_mod_T_square_2}, the claim follows. 
    \end{proof}
    
    Now, for all $\mfp\in \Omega_{A}$, the only hypothesis remaining to verify in Proposition~\ref{Pink_R_cond_T_adic_sur} is that the representation $\bar{\rho}_{\varphi,\mfp}$ is surjective, i.e., $M\equiv\GL_{2}(\F_{\mfp})\pmod{\mfp}$. Since
    $\det(\overline{M}_{\mfp})=\F_{\mfp}^{\times}$ (cf. Proposition~\ref{det_of_mfp_adic_rep_sujective_for_all}), the claim follows if we show $\SL_{2}(\F_{\mfp})\subseteq \overline{M}_{\mfp}$. 
    
    By Lemma~\ref{unipotent_contained_in_image_of_inertia_for_all_mfp}, the group $\overline{M}_{\mfp}$ contains a Sylow-$p$ subgroup of $\GL_{2}(\F_{p})$. By Lemma~\ref{zyw_sl_{2}_F}, we have either $\SL_2(\F_\mfp)\subseteq\overline{M}_{\mfp}$ or the action of $\overline{M}_{\mfp}$ on $\varphi[\mfp]\cong \F_{\mfp}^2$ is reducible, which is ruled out by the following proposition.

    \begin{prop}
        \label{irreducible_rep_for_all_mfp}
        Let $\varphi$ be as in Theorem~\ref{for_all_mfp_adic}. Then $\overline{M}_{\mfp}$ acts irreducibly on $\varphi[\mfp]$, i.e., the $\F_{\mfp}[G_F]$-module $\varphi[\mfp]$ is irreducible for all $\mfp\in\Omega_{A}$.
    \end{prop}
    \begin{proof}
        On the contrary, $\varphi[\mfp]$ is reducible for some $\mfp \in \Omega_A$. Then, there is a proper $\F_{\mfp}[G_{F}]$-submodule $X$ of $\varphi[\mfp]$ such that $X$ has $\F_{\mfp}$-dimension $1$. Hence, there exists a basis, say $\{v_1,v_2\}$, of $\varphi[\mfp]$, such that the action of $G_F$ on $\varphi[\mfp]$ is of the form $\psmat{\chi}{*}{0}{\chi^\prime}$, where $\chi, \chi^\prime:G_{F}\ra \F_\mfp^{\times}$ are characters. 
		
		$\bullet$ Assume $\mfp=\mfl$. Since $(\mfl, r_1) \in \Lambda_A$, there exists $c \in \F_q$ such that the polynomial $x^2-r_{1}x+(T-c)$ is irreducible over $\F_{\mfl}$. Note that $\varphi$ has  good reduction at $(T-c)$, and hence $\rho_{\varphi,\mfl}$ is unramified at $(T-c)$, because $\mfl \neq (T-c) $. Now, arguing as in the proof of Proposition~\ref{irreducible_rep_at_mfp}, the characteristic polynomial of $\bar{\rho}_{\varphi,\mfl}(\Frob_{(T-c)})$ is $x^2-r_1x+\overline{(T-c)}\in \F_{\mfl}[x]$. Since $\varphi[\mfl]$ is reducible, the polynomial $x^2-r_{1}x+(T-c)$ is reducible over $\F_{\mfl}$, which is a contradiction.

        $\bullet$ Assume $\mfp\neq\mfl$. Inspired by~\cite{Che22}, we proceed with the proof. We first need to understand the ramification behaviour of the characters $\chi$ and $\chi^\prime$.
	\begin{lem}
		The characters $\chi$ and $\chi^\prime$ are unramified at all $\lambda\in \Omega_{A}\setminus\{\mfp\}$. One of these two characters is unramified at all $\lambda\in \Omega_{A}$.
	\end{lem}
	\begin{proof}
		We first show that the characters $\chi$ and $\chi^\prime$ are unramified at all non-zero prime ideals $\lambda\neq \mfp$.
		\begin{itemize}
			\item $\lambda=\mfl$: By Lemma~\ref{unipotent_contained_in_image_of_inertia_for_all_mfp}, we have $|\bar{\rho}_{\varphi, \mfp}(I_\lambda)|=q^{\deg_{T}(\mfp)}$. But the characters $\chi$ and $\chi^\prime$ takes values in $\F_\mfp^{\times}$, whose cardinality, $q^{\deg_{T}(\mfp)}-1$, is relatively prime to $q$. Therefore, $\chi(I_\lambda)=\chi^\prime(I_\lambda)=1$.
			\item $\lambda\neq \mfl$: Since $\varphi$ has a good reduction at $\lambda$, $\bar{\rho}_{\varphi, \mfp}$ is unramified at $\lambda$ and hence $\chi$ and $\chi^\prime$.
		\end{itemize}
		We now show that one of these two characters is unramified at $\mfp$. Since $\varphi$ has rank $2$ and good reduction at $\mfp \neq \mfl$, the height $H_{\mfp}$ of $\varphi$ modulo $\mfp$ satisfies $H_{\mfp} \in \{1, 2\}$. Thus, it suffices to consider the following two cases:
		\begin{itemize}
			\item Suppose $H_{\mfp}=1$. Since the submodule $\varphi[\mfp]^{0}\cong \F_{\mfp}$, $\varphi[\mfp]^{0}$ is a $1$-dimensional $\F_{\mfp}[G_{F}]$-submodule of $\varphi[\mfp]$. Suppose $V_\chi=\spn\{v_{1}\}$ is an $\F_{\mfp}[G_{F}]$-submodule of $\varphi[\mfp]$ corresponding to $\chi$. Then $V_\chi=\varphi[\mfp]^{0}$ or $V\cap \varphi[\mfp]^{0}=0$.
			
			\begin{itemize}
				\item If $V_\chi=\varphi[\mfp]^{0}$, then $I_\mfp$ acts on
				$\varphi[\mfp]/\varphi[\mfp]^0\cong \bar{\varphi}[\mfp]$ by $\chi^\prime$. Now, by Proposition~\ref{PR_09}, $\chi^\prime(\sigma)=1$ for all $\sigma\in I_{\mfp}$.
				
				\item If $V_\chi \cap \varphi[\mfp]^{0}=0$, then $V_\chi \hookrightarrow \varphi[\mfp]/\varphi[\mfp]^0\cong \bar{\varphi}[\mfp]$ as a $\F_{\mfp}[I_{\mfp}]$-submodule. Consequently, by Proposition~\ref{PR_09}, we have $\chi(\sigma)=1$ for all $\sigma\in I_{\mfp}$.
			\end{itemize}
			
			\item Suppose $H_{\mfp}=2$. We show that this case cannot arise. If not, we have $\varphi[\mfp]^{0}\cong \varphi[\mfp]$. By Corollary~\ref{ch_PR09}, $\varphi[\mfp]$ is an irreducible $\F_{p}[I_{\mfp}]$-module, hence irreducible as a $\F_{p}[G_F]$-module, which is a contradiction to our assumption that $\varphi[\mfp]$ is a reducible $\F_{\mfp}[G_{F}]$-module.
		\end{itemize}
	\end{proof}
    For any $\widetilde{\mfp}\in \Omega_A\setminus\{\mfp,\mfl\}$, the Drinfeld $A$-module $\varphi$ has good reduction at $\widetilde{\mfp}$, hence $\rho_{\varphi,\mfp}$ is unramified at $\widetilde{\mfp}$. Set $a_{\widetilde{\mfp}}(\varphi) := \trace\big(\rho_{\varphi,\mfp}(\Frob_{\widetilde{\mfp}})\big) \in A$. The following lemma describes $a_{\widetilde{\mfp}}(\varphi) \pmod \mfp$ in terms of the characters $\chi$ and $\chi'$. The proof is similar to that of~\cite[Lemma~5.4]{Zyw11} for $\mfl=(T)$, and we avoid repeating it.
	
	\begin{lem} \label{trace_chi_chi'}
		Let $\mfp\in \Omega_{A}$ be such that $\varphi[\mfp]$ is a reducible $\F_{\mfp}[G_F]$-module. Then, there is a $\zeta\in \F_\mfp^{\times}$ such that for any monic irreducible polynomial $\widetilde{\mfp}\in A \setminus \{\mfp, \mfl \}$, we have
		$ a_{\widetilde{\mfp}}(\varphi)\equiv \zeta^{-\deg_{T}(\widetilde{\mfp})} \widetilde{\mfp} +\zeta^{\deg_{T}(\widetilde{\mfp})}\pmod{\mfp}$.
	\end{lem} 
    \begin{lem}
       \label{trace_T-c}
		Let $ \widetilde{\mfp}= (T-d)\in\Omega_{A}\setminus\{\mfp,\mfl\}$ for some $d\in \F_{q}$. Then $a_{\widetilde{\mfp}}(\varphi) = r_{1}$.
	\end{lem}
	\begin{proof}
		Arguing as in the proof of Proposition~\ref{irreducible_rep_at_mfp}, we get $a_{\widetilde{\mfp}}(\varphi)=g_{1}(d)=r_1$, since $g_1\equiv r_1\pmod{T^q-T}$.
	\end{proof}
    %Now, we are in a position to complete the proof of Proposition~\ref{Image-irreducible}.
    Since $q\geq 5$, there exist distinct $d_{1},d_{2}\in \F_{q}$ such that $(T-d_{i})\in\Omega_{A}\setminus\{\mfp,\mfl\}$ for $i=1,2$. By Lemmas \ref{trace_chi_chi'} and \ref{trace_T-c} with $\widetilde{\mfp}_i=T-d_{i}$ for $i=1,2$, we get
    \begin{equation*}
        r_{1} \equiv \zeta^{-1}(T-d_{i})+\zeta \pmod{\mfp}
    \end{equation*}
	this implies that $T - \zeta r_{1}+\zeta^2 \equiv d_{i}\pmod{\mfp}$ for distinct $d_{1},d_{2}\in \F_{q}$, which is a contradiction since $\deg_{T}(\mfp) \geq 1$. We are done with the proof of Proposition~\ref{irreducible_rep_for_all_mfp}.
	\end{proof}

    For all $\mfp\in \Omega_{A}$, the image $M_{\mfp}:=\rho_{\varphi,\mfp}(G_F)$ satisfies the hypotheses of Proposition~\ref{Pink_R_cond_T_adic_sur}, we complete the proof of Theorem~\ref{for_all_mfp_adic}.

	\section{Surjectivity of the adelic Galois representations}
	In this section, for a fixed Drinfeld $A$-module $\varphi$ of rank $2$ over $F$ with $\varphi_{T}\in A\{\tau\}$, whose coefficients satisfy certain congruence and valuation conditions, we show that the associated adelic Galois representation $\rho_{\varphi}$ is surjective. Define $\widehat{\Lambda}_{A}:=\left\{(\mfl,r_{1})\in\Lambda_{A}: r_1\in \F_{q}^{\times} \right\}.$ We now state the main theorem of this section.
	
	\begin{thm}
	\label{adelic_surjectivity}
		Let $q=p^e\geq 5$ be an odd prime power and $(\mfl,r_{1})\in \widehat{\Lambda}_{A}$. Let $\varphi$ be a Drinfeld $A$-module of rank $2$ over $F$ defined by
        \begin{equation*}
            \varphi_{T}=T+g_{1}\tau-\mfl^{q-1}\tau^2 \in A\{\tau\},
        \end{equation*}
        with $g_1\equiv r_1\pmod{T^q-T}$ and $\nu_{\mfl}(g_1)=0$. Then the adelic Galois representation \begin{equation*}
            \rho_{\varphi}:G_F\longrightarrow \GL_{2}(\widehat{A})
        \end{equation*}
		is surjective.
	\end{thm}
    By Lemma~\ref{Lambda_A_is_non_empty} and Proposition~\ref{infinitely_many_g_1}, there are infinitely many examples satisfying Theorem~\ref{adelic_surjectivity}.
    
    Although the proof of Theorem~\ref{adelic_surjectivity} is a verification of the abstract mechanism mentioned in~\cite[\S 6]{Zyw11}, a notable difference is that we give a detailed proof of Proposition~\ref{mod_a_surjectivity}. We first recall a lemma of Zywina (cf.~\cite[Lemma A.7]{Zyw11}). 
    \begin{lem}\label{lemma_A_7_zyw}
        Let $\mfl_1, \mfl_2 \in \Omega_A$ be distinct, and set $\mfa=\mfl_{1}\mfl_{2}$. Let $H$ be a subgroup of $\GL_{2}(A/\mfa)$ for which the following hold:
			\begin{enumerate}
				\item $\det(H)=(A/\mfa)^{\times}$;
				\item For $i=1,2$, the projections $p^\prime_{i}:H^\prime\ra \SL_{2}(\F_{\mfl_{i}})$ are surjective, where $H^\prime=H\cap \SL_{2}(A/\mfa)$;
				\item the subring of $A/\mfa$ generated by the set
                \begin{equation*}
                    \mathcal{S}:=\left\{\tr(h)^2/\det(h)|h\in H\right\}\cup\left\{\det(h)/\tr(h)^2|h\in H\ \text{with}\ \tr(h)\in (A/\mfa)^{\times}\right\}
                \end{equation*}
				is exactly $A/\mfa$.
			\end{enumerate}
        Then $H=\GL_{2}(A/\mfa)$.
    \end{lem}

    \begin{prop}
       \label{mod_a_surjectivity}
        Let $\varphi$ be as in Theorem~\ref{adelic_surjectivity}. Let $\mfl_1, \mfl_2 \in \Omega_A$ be distinct and set $\mfa=\mfl_{1}\mfl_{2}$. Then, the associated mod-$\mfa$ Galois representation is surjective, i.e., $\bar{\rho}_{\varphi,\mfa}(G_{F})=\GL_{2}(A/\mfa)$.
    \end{prop}
    \begin{proof}
  We show that $H:=\bar{\rho}_{\varphi,\mfa}(G_{F})$ satisfies the three hypotheses of Lemma~\ref{lemma_A_7_zyw}. By~\cite[Theorem 3.7.1(1)]{Pap23} and Proposition~\ref{Hayes}, arguing as in the proof of Proposition~\ref{det_of_mfp_adic_rep_sujective}, property $(1)$ follows. By Theorem~\ref{for_all_mfp_adic}, for $i=1,2$, the representations	$\bar{\rho}_{\varphi,\mfl_{i}}$ are surjective, hence $\bar{\rho}_{\varphi,\mfl_{i}}([G_{F},G_{F}])= [\bar{\rho}_{\varphi,\mfl_{i}}(G_F), \bar{\rho}_{\varphi,\mfl_{i}}(G_F)]= [\GL_{2}(\F_{\mfl_{i}}), \GL_{2}(\F_{\mfl_{i}})]= \SL_{2}(\F_{\mfl_{i}})$. Therefore, property $(2)$ follows from $\bar{\rho}_{\varphi,\mfa} ([G_{F},G_{F}]) \subseteq H^\prime$.

        For property~$(3)$, let $\langle \mcS \rangle $ be the subring of $A/\mfa$ generated by $\mcS$, and define
        \begin{equation*}
            C:=\{d\in\F_q:(T-d)\notin\{\mfl,\mfl_1,\mfl_2\}\}.
        \end{equation*}
        Since at most three values of $d$ are excluded, we have $|C|\ge q-3.$ For every $d\in C$, the characteristic polynomial of $\bar{\rho}_{\varphi,\mfa}(\Frob_{(T-d)})$ is congruent to $x^2-r_1x+(T-d)$ modulo $\mfa$ (cf. the proof of Proposition~\ref{irreducible_rep_at_mfp}). Therefore
        \begin{equation*}
            \det(\bar{\rho}_{\varphi,\mfa}(\Frob_{(T-d)}))/\tr(\bar{\rho}_{\varphi,\mfa}(\Frob_{(T-d)}))^2\equiv \frac{T-d}{r_1^2}\pmod{\mfa}.
        \end{equation*}
        Since $\tr(\bar{\rho}_{\varphi,\mfa}(\Frob_{(T-d)}))= r_1 
        \in\F_q^\times \subseteq (A/\mfa)^\times$, by definition of $\mcS$, we get $\frac{\bar{T}-d}{r_1^2}\in \mcS$, where $\bar{T}$ denotes $T \pmod \mfa$. In particular, for $d_1,d_2\in C$, we have $\frac{d_2-d_1}{r_1^2}\in \langle \mcS \rangle.$
        
        Let $W$ be the additive subgroup of $\F_q$ generated by all differences $d_2-d_1$ with $d_1,d_2\in C$. For any fixed $d_0\in C$, we have $C\subseteq d_0+W,$ and therefore $|W|\ge |C|\ge q-3$. Since $W$ is an additive subgroup of $\F_q$, it is an $\F_p$-subspace. If $W \neq\F_q$, then $|W|\le p^{e-1}=\frac qp\le\frac q3$, because $q=p^e$ is an odd prime power. On the other hand, for $q \geq 5$, we have $q-3>\frac{q}{3}$, which is a contradiction to $|W| \ge q-3.$. Thus $W=\F_q.$

        Recall that, for every $d_1,d_2\in C$, $\frac{d_2-d_1}{r_1^2}\in \langle \mcS \rangle$. Hence, we get $r_1^{-2}W\subseteq \langle \mcS \rangle.$ Since $W=\F_q$ and $r_1\in\F_q^\times$,
        we have $\F_q\subseteq \langle \mcS \rangle.$ This implies that 
        $\bar{T}=d+r_1^2\frac{\bar{T}-d}{r_1^2}\in \langle \mcS \rangle$,
        where $d\in C$ is arbitrary. Thus, $\langle \mcS \rangle$ contains both $\F_q$ and $\bar{T}$, and hence $\langle \mcS \rangle=\F_q[\bar T]=A/\mfa.$ Therefore, property~$(3)$ follows. By Lemma~\ref{lemma_A_7_zyw}, we complete the proof of the proposition.
    \end{proof}
    Now, arguing in a similar way as in ~\cite[Lemma 6.3]{Zyw11} and using
    Proposition~\ref{mod_a_surjectivity} implies the following:
    \begin{lem}
    \label{lemma_6_3_zyw}
        Let $\varphi$ be as in Theorem~\ref{adelic_surjectivity}. Let $\mfl_1, \mfl_2 \in \Omega_A$ be distinct. Define
        \begin{equation*}
            \rho:G_{F}\longrightarrow \GL_{2}(A_{\mfl_1})\times\GL_{2}(A_{\mfl_2}),\quad \sigma\mapsto (\rho_{\varphi,\mfl_1}(\sigma),\rho_{\varphi,\mfl_2}(\sigma)).
        \end{equation*}
        Then $\rho([G_{F},G_{F}])=\SL_{2}(A_{\mfl_1})\times\SL_{2}(A_{\mfl_2})$.
    \end{lem}
    
    We are finally in a position to give the proof of Theorem~\ref{adelic_surjectivity}.
	\begin{proof}[Proof of Theorem~\ref{adelic_surjectivity}]
		By~\cite[Theorem 3.7.1(1)]{Pap23} and Proposition~\ref{Hayes}, arguing as in the proof of Proposition~\ref{det_of_mfp_adic_rep_sujective}, we get 
		$\det\rho_{\varphi}(G_{F})=\rho_{C}(G_{F})={\widehat{A}}^{\times}$. 
        
        Now, it is enough to show $\rho_{\varphi}([G_{F}, G_{F}])=\SL_{2}(\widehat{A})$ and this is equivalent to $\bar{\rho}_{\varphi,\mfa}([G_{F},G_{F}])=\SL_{2}(A/\mfa)\cong \prod_{i}^{}\SL_{2}(A/{\mfl_{i}^{n_{i}}}), $ for every non-zero ideal $ \mfa=\mfl_{1}^{n_{1}}\mfl_{2}^{n_{2}}\cdots \mfl_{k}^{n_{k}}$ of $A$. By~\cite[Lemma A.3]{Zyw11}, each $\SL_{2}(A/{\mfl_{i}^{n_{i}}})$ has no non-trivial abelian quotient. By Lemma~\ref{lemma_6_3_zyw}, each projection
        \begin{equation*}
            \bar{\rho}_{\varphi,\mfl_{i}^{n_i}\mfl_{j}^{n_j}}:[G_{F},G_{F}]\longrightarrow \SL_{2}(A/{\mfl_{i}^{n_{i}}})\times \SL_{2}(A/{\mfl_{j}^{n_{j}}})
        \end{equation*} 
        is surjective for $1\leq i\leq j\leq k$. Now, by~\cite[Lemma 5.2.2]{Rib76}, we have $\bar{\rho}_{\varphi,\mfa}([G_{F},G_{F}])=\SL_{2}(A/\mfa)$ for all non-zero ideal $\mfa$ of $A$. This completes the proof of the theorem.
	\end{proof}
    The following corollaries are consequences of 
    Theorem~\ref{adelic_surjectivity}.
    \begin{cor}
    \label{nonzero_F_rational}
        Let $\varphi$ be as in Theorem~\ref{adelic_surjectivity}. Then $\varphi$ has no non-zero $F$-rational torsion point.
    \end{cor}
    \begin{proof}
        For any non-zero ideal $\mfa$ of $A$, we define the set of $F$-rational $\mfa$-torsion points $\varphi[\mfa](F)$ as the set of elements in $\varphi[\mfa]$ fixed by the action of $G_F$. Since $\varphi[\mfa]\cong(A/\mfa)^2$ and the mod-$\mfa$ representation $\bar{\rho}_{\varphi,\mfa}$ is surjective, by Theorem~\ref{adelic_surjectivity}, elements of $\varphi[\mfa](F)$ correspond to elements $(u,v)\in (A/\mfa)^2$ fixed by the action of $\GL_2(A/\mfa)$. The group $\GL_2(A/\mfa)$ fixes no non-zero element in $(A/\mfa)^2$, so we get $\varphi[\mfa](F)=\{0\}$. 
        Thus, $\varphi$ has no non-zero $F$-rational torsion point. 
    \end{proof}
     Analogous to Corollary~\ref{cor_padic_inverse_galois}, Theorem~\ref{adelic_surjectivity} gives the following corollary: 
    
    \begin{cor}
    \label{cor:adelic_inverse_galois}
        Let $\varphi$ be as in Theorem~\ref{adelic_surjectivity}. 
        If $F_{\mathrm{tors}}$ denotes the compositum of the torsion fields $F(\varphi[\mfa]),$ where $\mfa$ runs over all non-zero ideals of $A$, then we get $Gal(F_{\mathrm{tors}}/F) \simeq \GL_2(\widehat A).$ In particular, for every non-zero ideal $\mfa\subseteq A$, we have $\Gal(F(\varphi[\mfa])/F)\simeq \GL_2(A/\mfa)$.
    \end{cor}
    Finally, we conclude this article with some remarks that highlight comparisons between our results and those of~\cite{Zyw11}, \cite{Zyw25}.
\subsection{Comparison with~\cite{Zyw11}}
    In~\cite[Theorem 1.2]{Zyw11}, Zywina constructed a Drinfeld $A$-module of rank $2$ over $F$ with surjective adelic Galois representation. In this article, we show that for every $(\mfl,r_1)\in\widehat{\Lambda}_{A}$, there exist infinitely many Drinfeld $A$-modules of rank $2$ over $F$ with surjective adelic Galois representations. In particular, our construction recovers the example considered in~\cite[Theorem 1.2]{Zyw11} as it corresponds to $((T),1)\in\widehat{\Lambda}_{A}$.

\subsection{Comparison with~\cite{Zyw25}}
    In the proof of~\cite[Theorem~9.1]{Zyw25}, Zywina introduces four subsets $\mathcal{R}$, $\mathcal{S}_m$, $\mathcal{T}_m$, and $\mathcal{U}_m$ of $A^2$ for $m \ge 2$ (cf.~\cite[p.~22]{Zyw25}), and shows that if $(g_1,g_2)\in A^2$ lies in $\mathcal{R}\cap\mathcal{S}_m\cap\mathcal{T}_m\cap\mathcal{U}_m$, then the associated Drinfeld $A$-module of rank $2$ over $F$ defined by $\varphi_T = T + g_1\tau + g_2\tau^2 \in A\{\tau\}$ has surjective $\mfp$-adic Galois representations for all $\mfp \in \Omega_A$. In particular, the set $\mathcal{R}$ consists of pairs $(g_1,g_2)\in A^2$ for which there exist at least two distinct non-zero prime ideals $\mfp$ of $A$ with $\deg_T(\mfp) > 1$, $\nu_{\mfp}(g_1)=0$ and $\nu_{\mfp}(g_2)=1$. In contrast, the Drinfeld $A$-modules of rank $2$ over $F$ constructed in this article (Theorem~\ref{for_all_mfp_adic}) lie outside of $\mathcal{R}$, and hence we provide examples that are not covered by the framework in~\cite{Zyw25}.

	\section{Appendix}
    In this appendix, for any Drinfeld $A$-module $\varphi$ of rank $r$ over $K$, where $K$ is an $A$-field, we shall explicitly write down the coefficients of $\varphi_{T^j}$, for $j \geq 1$, in terms of the coefficients of $\varphi_T$. These explicit calculations may have appeared elsewhere; for the convenience of the reader, we reproduce the argument here.
    \begin{prop}
    \label{expression_for_varphi_a}
        Let $\varphi$ be a Drinfeld $A$-module of rank $r$ over $K$, defined by $\varphi_{T}= g_0+g_1\tau+\cdots+g_r\tau^r$, where $g_0:=\gamma(T)$. For $j\geq 1$, we have $\varphi_{T^j}=\varphi^{j}_{T}=\sum_{k=0}^{jr}c_{k}^{(j)}\tau^k$, where
    \begin{equation}
    \label{expression_of_vaphi_T_cap_n_with_explicit_coefficients}
       	c_{k}^{(j)}=\sum_{\substack{0\leq i_1,i_2,\ldots,i_j\leq r \\ i_1+i_2+\cdots+i_j=k}}^{}g_{i_{1}}g_{i_{2}}^{q^{i_{1}}}g_{i_{3}}^{q^{i_1+i_2}}\cdots g_{i_{j}}^{q^{i_1+i_2+\cdots+i_{j-1}}}\quad\text{and}\quad c^{(0)}_{0} :=1.
    \end{equation}
    \end{prop}
    \begin{proof} 
        For $j=0$, $\varphi_{T^0}=\varphi_{1}=\gamma(1)=1=c^{(0)}_{0}$, the formula in~\eqref{expression_of_vaphi_T_cap_n_with_explicit_coefficients} holds. Similarly, for $j=1$, $c_{k}^{(1)}=g_k$ for $k=0,1,\ldots,r$, again the formula in~\eqref{expression_of_vaphi_T_cap_n_with_explicit_coefficients} holds. Now, we shall prove the proposition by mathematical induction on $j$.

        Assume that the formula holds for some $j \geq 1$, i.e., $\varphi_{T^j}=\varphi^{j}_{T}=\sum_{m=0}^{jr}c_{m}^{(j)}\tau^m$ and $c_{m}^{(j)}$ as in \eqref{expression_of_vaphi_T_cap_n_with_explicit_coefficients}. Now $\varphi_{T^{j+1}}=\varphi_{T^j}\varphi_{T}$, so
    \begin{align*}
        \varphi_{T^{j+1}}&=\Big(\sum_{m=0}^{jr}c_{m}^{(j)}\tau^m\Big)\big(g_0+g_1\tau+\cdots+g_r\tau^r\big)\\
        &=\sum_{m=0}^{jr}c_{m}^{(j)}g_0^{q^m}\tau^m+\sum_{m=0}^{jr}c_{m}^{(j)}g_1^{q^m}\tau^{m+1}+\cdots+\sum_{m=0}^{jr}c_{m}^{(j)}g_r^{q^m}\tau^{m+r}\\
        &=\sum_{m=0}^{jr}c_{m}^{(j)}\big(\sum_{i=0}^{r}g_i^{q^m}\tau^{i}\big)\tau^{m}=\sum_{m=0}^{jr}\sum_{i=0}^{r}c_{m}^{(j)}g_i^{q^m}\tau^{m+i}.
    \end{align*}
    Hence, the coefficient of $\tau^k$ is $\sum_{\substack{m+i=k\\ 0\leq m\leq jr,\ 0\leq i\leq r}}c_{m}^{(j)}g_{i}^{q^m},$ and substituting the inductive expression for $c_{m}^{(j)}$, we get
    \begin{equation*}
        c_{k}^{(j+1)}=\sum_{\substack{0\leq i_1,i_2,\ldots,i_j,i_{j+1}\leq r \\ i_1+i_2+\cdots+i_j+i_{j+1}=k}}^{}g_{i_{1}}g_{i_{2}}^{q^{i_{1}}}g_{i_{3}}^{q^{i_1+i_2}}\cdots g_{i_{j}}^{q^{i_1+i_2+\cdots+i_{j-1}}}g_{i_{j+1}}^{q^{i_1+i_2+\cdots+i_{j-1}+i_j}}.
    \end{equation*}
    This completes the proof of the proposition by induction.
    \end{proof} 
    \begin{cor}
    \label{expression_of_varphi_a_with_explicit_coefficients}
        Let $a=\sum_{j=0}^{n}a_jT^j\in A$, so $\varphi_{a}=a(\varphi_{T})=\sum_{j=0}^{n}a_j\varphi^j_{T}$, substituting the expansion of $\varphi^j_{T}$, we get
        \begin{equation*}
            \varphi_{a}=\sum_{j=0}^{n}a_{j}\sum_{k=0}^{jr}
            c_{k}^{(j)}\tau^k =\sum_{k=0}^{nr} \Big(\sum_{j=\big\lceil \tfrac{k}{r} \big\rceil}^{n}a_{j}c_{k}^{(j)}\Big)\tau^k. 
        \end{equation*}
    \end{cor}
\section*{Acknowledgments}
    The first author's research was partially supported by the ANRF CRG Grant CRG/2023/003136. The authors are grateful to the anonymous referee for providing many valuable mathematical comments and suggestions, which significantly improved our article.

	\bibliographystyle{plain, abbrv}

\end{document}